\pdfoutput=1
\RequirePackage{ifpdf}
\ifpdf 
\documentclass[pdftex]{sigma}
\else
\documentclass{sigma}
\fi

\usepackage{mathrsfs}
\usepackage[all,cmtip]{xy}

\newtheorem{theo}{Theorem}[section]
\newtheorem{lem}[theo]{Lemma}
\newtheorem{propo}[theo]{Proposition}
\newtheorem{cor}[theo]{Corollary}

\theoremstyle{definition}
\newtheorem{defi}[theo]{Definition}
\newtheorem{ex}[theo]{Example}
\newtheorem{remm}[theo]{Remark}

\numberwithin{equation}{section}

\def\nn{\nonumber}

\def\Hom{\mathrm{Hom}}

\def\Der{\mathrm{Der}}

\def\Con{\mathrm{Con}}

\def\dd{\mathrm{d}}
\def\du{\mathrm{d_u}}

\def\id{\mathrm{id}}

\def\coH{{\mathrm{co}H}}

\def\OO{\mathcal{O}}

\def\bbZ{\mathbb{Z}}

\def\bbR{\mathbb{R}}

\def\bbC{\mathbb{C}}
\def\bbT{\mathbb{T}}
\def\bbS{\mathbb{S}}
\def\1{\mathbf{1}}
\def\flip{\mathrm{flip}}
 \def\sw#1{{\sb{\underline{#1}}}}

 \def\can{\mathrm{can}}
 \def\rmod#1{\mathscr{M}{}\!\!_#1}
 \def\rcom#1{\mathscr{M}^#1}
 \def\lmod#1{{}_#1\mathscr{M}}
 \def\lcom#1{{}^#1\!\!\mathscr{M}}
 \def\lcomf#1{{}^#1\!\!\mathscr{M}{}\!_{\mathrm{fin}}}
 \def\proj#1{{}_#1\mathscr{P}}
 \def\ver{\mathrm{ver}}
 \def\ad{\mathrm{ad}}
 \def\s{\mathsf{s}}

\begin{document}
\allowdisplaybreaks

\newcommand{\arXivNumber}{1811.10913}

\renewcommand{\thefootnote}{}

\renewcommand{\PaperNumber}{008}

\FirstPageHeading

\ShortArticleName{On the Relationship between Classical and Deformed Hopf Fibrations}

\ArticleName{On the Relationship between Classical\\ and Deformed Hopf Fibrations\footnote{This paper is a~contribution to the Special Issue on Noncommutative Manifolds and their Symmetries in honour of Giovanni Landi. The full collection is available at \href{https://www.emis.de/journals/SIGMA/Landi.html}{https://www.emis.de/journals/SIGMA/Landi.html}}}

\Author{Tomasz BRZEZI{\'N}SKI~$^{\dag\ddag}$, James GAUNT~$^\S$ and Alexander SCHENKEL~$^\S$}

\AuthorNameForHeading{T.~Brzezi{\'n}ski, J.~Gaunt and A.~Schenkel}

\Address{$^\dag$~Department of Mathematics, Swansea University, Swansea University Bay Campus,\\
\hphantom{$^\dag$}~Fabian Way, Swansea SA1 8EN, UK}
\EmailD{\href{mailto:T.Brzezinski@swansea.ac.uk}{T.Brzezinski@swansea.ac.uk}}
\Address{$^\ddag$~Faculty of Mathematics, University of Bia{\l}ystok,\\
\hphantom{$^\ddag$}~K.~Cio{\l}kowskiego 1M, 15-245 Bia{\l}ystok, Poland}

\Address{$^\S$~School of Mathematical Sciences, University of Nottingham,\\
\hphantom{$^\S$}~University Park, Nottingham NG7 2RD, UK}
\EmailD{\href{mailto:james.gaunt@nottingham.ac.uk}{james.gaunt@nottingham.ac.uk}, \href{mailto:alexander.schenkel@nottingham.ac.uk}{alexander.schenkel@nottingham.ac.uk}}

\ArticleDates{Received September 17, 2019, in final form February 17, 2020; Published online February 23, 2020}

\Abstract{The $\theta$-deformed Hopf fibration $\mathbb{S}^3_\theta\to \mathbb{S}^2$ over the commutative $2$-sphere is compared with its classical counterpart. It is shown that there exists a natural isomorphism between the corresponding associated module functors and that the affine spaces of classical and deformed connections are isomorphic. The latter isomorphism is equivariant under an appropriate notion of infinitesimal gauge transformations in these contexts. Gauge transformations and connections on associated modules are studied and are shown to be sensitive to the deformation parameter. A homotopy theoretic explanation for the existence of a close relationship between the classical and deformed Hopf fibrations is proposed.}


\Classification{81T75; 16T05}

\begin{flushright}
\em To Gianni on the occasion of his 60th birthday
\end{flushright}

\renewcommand{\thefootnote}{\arabic{footnote}}
\setcounter{footnote}{0}

\section{Introduction and summary}\label{sec:intro}
The Hopf fibration $\bbS^3 \to \bbS^2$ is the
prime example of a non-trivial principal $U(1)$-bundle
over the $2$-sphere. From an algebraic perspective,
it can be described as a faithfully-flat Hopf--Galois extension, or equivalently as a principal comodule algebra, consisting of the algebra $A= \mathcal{O}\big(\bbS^3\big)$ of functions on $\bbS^3$ together with the canonically induced coaction $\delta\colon A\to A\otimes H$ of the Hopf algebra $H = \mathcal{O}(U(1))$ of functions on the structure group $U(1)$. Due to its origin in ordinary geometry, this Hopf--Galois extension is special in the sense that the total space algebra $A$, the structure Hopf algebra $H$ and consequently the base space algebra $B := A^{\coH} \cong \mathcal{O}\big(\bbS^2\big)$ are commutative.

As Hopf--Galois theory does not require commutative algebras,
it provides a natural framework in which to study noncommutative
generalizations of principal bundles. In particular, there exists
a $1$-parameter family of deformations of the Hopf fibration
$\bbS^3 \to \bbS^2$, where the total space algebra is deformed to the
Connes--Landi $3$-sphere $A_\theta = \mathcal{O}\big(\bbS^3_\theta\big)$
and the structure Hopf algebra $H$ and base space algebra $B$
remain undeformed. It is important to emphasize that even though
the base space algebra $B$ and structure Hopf algebra $H$
are commutative, these examples are {\em not} commutative
principal bundles since the total space $A_\theta$
is a noncommutative algebra. Hence, one should expect certain,
potentially subtle, noncommutative geometry features in these examples.

In this work, we shall study in detail the
geometric structures on the deformed Hopf
fibrations $\bbS^3_\theta\to \bbS^2$ that are
relevant for gauge theory and compare those with
the corresponding structures on the classical Hopf fibration
$\bbS^3\to \bbS^2$. An interesting observation is that many, however not all,
of these geometric structures coincide for these examples
even though $\bbS^3_\theta\to \bbS^2$ and $\bbS^3\to \bbS^2$
are {\em not} isomorphic as Hopf--Galois extensions. This follows as
$A_\theta = \mathcal{O}\big(\bbS^3_\theta\big)$ is a noncommutative
algebra whilst $A = \mathcal{O}\big(\bbS^3\big)$ is commutative.
In more detail, we prove the following results:
\begin{enumerate}\itemsep=0pt
\item The associated module functors for the classical and deformed
Hopf fibrations are naturally isomorphic, i.e., the theory of associated modules
is insensitive to the deformation parameter $ \theta $. In physics terminology, this
means that we have the same matter fields on the classical and
deformed Hopf fibrations.
\item The affine spaces of connections, with respect to suitable K\"ahler-type differential
calculi, for the classical and deformed Hopf fibrations are isomorphic. It is also shown that this isomorphism is compatible with the action of infinitesimal gauge transformations.
In physics terminology, this means that we have the same gauge fields
on the classical and deformed Hopf fibrations.
\item In contrast to the previous two points,
the action of infinitesimal gauge transformations and connections
on associated modules does depend on the deformation parameter. Hence, it is different for the classical and deformed Hopf fibrations.
In physics terminology, this means that the coupling of gauge fields to matter
fields is sensitive to the deformation parameter.
\end{enumerate}

Even though our direct calculations were able to unravel these striking similarities
between the classical and deformed Hopf fibrations, they provide no
conceptual reason for why these two non-isomorphic Hopf--Galois extensions
should behave similarly in certain respects. As a~first step towards a more conceptual
explanation, we investigate our examples of Hopf--Galois extensions from the homotopy
theoretic perspective proposed by Kassel and Schneider in~\cite{KS}.
We shall show that the classical and deformed Hopf fibrations are homotopy equivalent,
however in a slightly different way as the one proposed by Kassel and Schneider.
In more detail, while the interval object in~\cite{KS} is modelled by
the polynomial algebra $\bbC[y]$ of the affine line,
we require an interval object that is modelled by a larger algebra
that also contains exponential functions.
This is related to the fact that the deformation parameter $\theta$,
which we would like to turn to zero by a homotopy equivalence,
enters the deformed Hopf fibration in an exponential form
$q={\rm e}^{2\pi {\rm i}\theta}$. We give some indications, however not a full proof,
that this homotopy equivalence could be the reason why the classical and
deformed Hopf fibrations have naturally isomorphic associated module functors.
A detailed study of these aspects, and in particular
of the interplay between homotopy equivalence and connections,
is beyond the scope of this paper. However, we hope to come back to this issue in a future work.

The outline of the remainder of this paper is as follows:
In Section \ref{sec:HopfGalois} we provide a brief review of
the theory of Hopf--Galois extensions and principal comodule algebras.
We also introduce the examples of interest in this work,
namely the classical Hopf fibration $\bbS^3\to \bbS^2$ and its
deformation $\bbS^3_\theta \to \bbS^2$ by a suitable family of $2$-cocycles.
In Section \ref{sec:associatedmodules} we describe
the associated module functors for both the classical and deformed
Hopf fibrations and prove that they are naturally isomorphic functors.
We would like to emphasize that this natural isomorphism is a specific feature
of our particular example and not a consequence of the general theory of $2$-cocycle
deformations, see Remark \ref{rem:nococycle}.
Section \ref{sec:connections} starts with a brief review of the theory
of Atiyah sequences and connections on principal comodule algebras.
We shall discuss both the case of universal differential calculi
and also the more general case of concordant differential calculi
on principal comodule algebras. We then describe in detail
the Atiyah sequence for universal and K\"ahler-type differential
calculi on the classical and deformed Hopf fibrations
and construct an isomorphism between the affine spaces of classical and
deformed connections. Again, we emphasize that this isomorphism results from
the specific example under discussion and not from the theory of $2$-cocycle deformations,
see Remark~\ref{rem:connectionbijection}. It is also shown that this isomorphism is
compatible with the action of infinitesimal gauge transformations.
In Section~\ref{sec:associatedcon} we study
gauge transformations and connections on
the associated modules of the classical and deformed Hopf fibrations
and show that these structures are sensitive to the deformation parameter.
In Section~\ref{sec:homotopy} we show that the classical and deformed Hopf fibrations
are in a suitable sense homotopy equivalent and give some indications
why this should imply the properties of associated modules described in Section~\ref{sec:associatedmodules}.

{\bf Notation and conventions:}
Throughout the bulk of the paper, by an algebra we mean an associative and unital algebra over a field $k$. Unless otherwise stated, $k=\bbC$, the field of complex numbers. The exception is Section~\ref{sec:homotopy} where algebras over commutative rings are also admitted.
The multiplication map in an algebra
$A$ is denoted by $\mu_A\colon A\otimes A\to A$ and the unit map by $\eta_A\colon k\to A$.
The unit element $\1_A\in A$, or simply $\1\in A$,
is obtained by evaluating the unit map on $1\in k$.
We denote the category of left $A$-modules by $\lmod A$
and that of right $A$-modules by $\rmod A$.
The full subcategory of finitely generated
projective left $A$-modules is denoted by $\proj A\subseteq \lmod A$.

The comultiplication in a Hopf algebra $H$ is denoted by $\Delta$ (or $\Delta_H$ if not
sufficiently clear from the context), the counit by $\epsilon$ (or $\epsilon_H$), and the
antipode by~$S$ (or~$S_H$). We always assume that~$S$ is a bijective map.
For comultiplications $\Delta\colon H\to H\otimes H$, right $H$-coactions $\delta\colon V\to V\otimes H$
and left $H$-coactions $\rho\colon V\to H\otimes V$, we use the following variant of Sweedler's notation (with suppressed summation)
\begin{gather*}
\Delta(h) = h\sw 1\otimes h\sw 2 , \qquad \delta(v) = v\sw 0\otimes v\sw 1 , \qquad
\rho(v) = v\sw{-1}\otimes v\sw 0 .
\end{gather*}
The category of left $H$-comodules is denoted by $\lcom H$ and that of right $H$-comodules by $\rcom H$.
The category of finite-dimensional left $H$-comodules is denoted by $\lcomf H$.

\section{Noncommutative Hopf fibrations}\label{sec:HopfGalois}

\subsection{Principal comodule algebras}\label{sec:principal}
The study of connections on noncommutative
principal bundles has been initiated in \cite{BrzMaj:gau} and
developed further by the introduction of {\em strong connections} in \cite{DabGro:str, Haj:str}.
This framework has been extended beyond Hopf algebras in \cite{BrzMaj:coa,BrzMaj:fac}
and then formalized in terms of {\em principal coalgebra extensions} in \cite{BrzHaj:Gal} and
{\em principal comodule algebras} in \cite{HajKra:pie}.
\begin{defi}\label{def:princ.com.alg}
Let $H$ be a Hopf algebra with bijective antipode.
A right $H$-comodule algebra $(A,\delta)$
is called a {\em principal comodule algebra}
if it admits a {\em strong connection}, i.e., a linear map $\ell\colon H \to A\otimes A$,
such that
\begin{subequations}\label{str.con}
\begin{alignat}{3}
\label{con.norm}
& \ell (\1_H) = \1_A\otimes \1_A \qquad && \mbox{(normalization)},&\\
\label{con.split}
& \mu_A \circ \ell = \eta_A\circ \epsilon_H \qquad &&\mbox{(splitting property)},&\\
\label{con.right}
&(\id \otimes \delta) \circ \ell = (\ell\otimes \id)\circ\Delta_H \qquad &&\mbox{(right colinearity)},& \\
\label{con.left}
&(\delta_S \otimes\id) \circ \ell = (\id\otimes \ell)\circ\Delta_H \qquad &&\mbox{(left colinearity)}, &
\end{alignat}
\end{subequations}
where $\delta\colon A\to A\otimes H$ is the right $H$-coaction and
$\delta_S\colon A\to H\otimes A$ is the associated left $H$-coaction defined by
\begin{gather*}
\delta_S := \big(S^{-1}\otimes \id\big)\circ \flip \circ \delta\colon \ a\longmapsto S^{-1}(a\sw 1)\otimes a\sw 0.
\end{gather*}
\end{defi}

In the context of noncommutative geometry, principal comodule algebras are interpreted as principal bundles. The algebra $A$ is the algebra of functions on the (noncommutative) total space, the Hopf algebra $H$ is the structure (quantum) group and the subalgebra of {\em coinvariants}
\begin{gather*}
B := A^{\coH} := \big\{a\in A\colon \delta(a) = a\otimes \1_H\big\}\subseteq A
 \end{gather*}
is the algebra of functions on the (noncommutative) base space.
Note that since $\delta$ is an algebra homomorphism, $B$ is indeed a subalgebra of~$A$.
The existence of a strong connection ensures that $A$ is a {\em Hopf--Galois extension}
of $B$, i.e., that the {\em canonical Galois map}
\begin{gather}\label{Galois}
\can := (\mu_A\otimes\id)\circ (\id\otimes \delta)\colon \ A\otimes_B A\longrightarrow A\otimes H ,\qquad
 a\otimes_B a^\prime\longmapsto a\, a^\prime \sw 0\otimes a^\prime\sw 1
 \end{gather}
is bijective. Explicitly, the inverse of the canonical Galois map is given by the composite of
\begin{gather*}
 \xymatrix{A\otimes H \ar[rr]^-{\id \otimes \ell} && A\otimes A\otimes A\ar[rr]^-{\mu_A\otimes \id}
 && A\otimes A \ar@{->>}[r] & A\otimes _B A.}
\end{gather*}
The Galois property encodes freeness of the action of the structure (quantum) group.

The existence of a strong connection also implies that $A$ is an
{\em $H$-equivariantly projective left $B$-module},
i.e., the restriction of the multiplication map to $B\otimes A$ has a
right $H$-comodule left $B$-module splitting. Explicitly,
\begin{gather*}
 \sigma:= (\mu_A\otimes \id)\circ (\id\otimes \ell)\circ \delta\colon \ A\longrightarrow B\otimes A .
\end{gather*}
In fact, a principal comodule algebra is the same as an $H$-equivariantly projective Hopf--Galois extension.
The projectivity property gives the notion of a principal comodule algebra full geometric meaning,
as it implies that $A$ admits a noncommutative connection in the sense of~\cite{CunQui:alg}. This is in perfect concord with Cartan's definition of a principal action of a compact Lie group~\cite{Car:pri}.
We return to these differential geometric aspects of principal comodule algebras in Section~\ref{sec:connections}.

Furthermore, a principal comodule algebra is the same as a {\em faithfully-flat
Hopf--Galois extension}, i.e., a Hopf--Galois extension $A$ of $B$ such that the tensor
product functor $(-)\otimes_B A\colon \rmod B \to \rmod A$ both preserves and reflects
exact sequences. This gives an algebraic geometry flavour to the notion of a principal
comodule algebra, as it leads to the faithfully flat descent property. As observed by H.-J.~Schneider, in one of the key results of Hopf--Galois theory \cite[Theorem~I]{Sch:pri},
if~$H$ admits an invariant integral (i.e., $H$ is {\em coseparable}) such as the Haar
measure on the coordinate algebra of a compact quantum group \cite{Wor:com},
then surjectivity of the canonical Galois map~\eqref{Galois} implies its injectivity
as well as faithful-flatness of $A$ as a left $B$-module.
An explicit construction of a strong connection, in the more general situation of
extensions by coalgebras, is given in~\cite{BegBrz:exp}.
As a consequence, Hopf--Galois extensions given by typical Hopf algebras that
feature in noncommutative geometry, and in particular those found in the present text,
are automatically principal comodule algebras.

The reader interested in studying further the meaning of principal comodule algebras
in classical geometry is encouraged to consult~\cite{BauHaj:non}.

\subsection{\label{subsec:classical}The classical Hopf fibration}
The Hopf fibration $\bbS^3\to\bbS^2$ can be described algebraically as follows.
The total space is given by the $\ast$-algebra $A = \mathcal{O}\big(\bbS^3\big)$
of functions on the algebraic $3$-sphere. Concretely, $A$
is the commutative $\ast$-algebra generated by
$z_1$ and $z_2$, modulo the $\ast$-ideal generated by
the $3$-sphere relation
\begin{gather*}
z_1^\ast z_1 + z_2^\ast z_2= \1 .
\end{gather*}
The structure group is described by the $\ast$-Hopf algebra
$H = \mathcal{O}(U(1))$ of functions on the algebraic
circle group $U(1)$. Concretely, $H$ is the commutative $\ast$-algebra generated
by $t$, modulo the $\ast$-ideal generated by the circle relation
\begin{gather*}
t^\ast t=\1 .
\end{gather*}
The coproduct, counit and antipode read as
\begin{gather*}
\Delta(t) = t\otimes t ,\qquad \epsilon(t) = 1 ,\qquad S(t) = t^\ast .
\end{gather*}
We endow $A$ with the structure of a right $H$-comodule
$\ast$-algebra by defining the right $H$-coaction $\delta\colon A\to A\otimes H$
on the generators as
\begin{subequations}\label{eqn:Hcoaction}
\begin{gather}
\delta(z_1) = z_1 \otimes t ,\qquad \delta(z_2)= z_2\otimes t .
\end{gather}
The compatibility condition $\delta \circ \ast = (\ast\otimes\ast) \circ\delta$
between the coaction and $\ast$-involution gives
\begin{gather}
\delta(z_1^\ast) = z_1^\ast \otimes t^\ast ,\qquad \delta(z_2^\ast)= z_2^\ast\otimes t^\ast .
\end{gather}
\end{subequations}
\begin{lem}\label{lem:2sphere}
The $\ast$-subalgebra of $H$-coinvariants
\begin{gather*}
B := A^{\coH} := \big\{a\in A\colon \delta(a) = a\otimes \1\big\}\subseteq A
\end{gather*}
is isomorphic to the $\ast$-algebra $\mathcal{O}\big(\bbS^2\big)$ of functions on the algebraic $2$-sphere.
\end{lem}
\begin{proof}We find that $A^{\coH}\subseteq A$
is generated as a $\ast$-algebra by
\begin{gather*}
z := 2 z_1 z_2^\ast\in A^{\coH} ,\qquad
x := z_1^\ast z_1 - z_2^\ast z_2\in A^{\coH} .
\end{gather*}
These generators satisfy the relations
\begin{gather*}
x^\ast =x ,\qquad z^\ast z + x^2 = \1 ,
\end{gather*}
hence $A^{\coH} \cong \mathcal{O}\big(\bbS^2\big)$.
\end{proof}

\begin{propo}\label{propo:HGcommutative}
The right $H$-comodule $\ast$-algebra $(A,\delta)$ described above
is a principal comodule algebra.
\end{propo}
\begin{proof}This is a special case of what was proven in \cite{BrzezinskiSitarz}.
We note in passing that a strong connection can be defined iteratively by
\begin{gather}\label{str.con.Hopf}
\ell (\1) = \1\otimes \1 , \qquad \ell (t^n) =
\begin{cases}
z_1^\ast \ell\left(t^{n-1}\right) z_1 + z_2^\ast \ell\left(t^{n-1}\right) z_2 ,& \text{for } n>0,\\
z_1 \ell\left(t^{n+1}\right) z^\ast_1 + z_2 \ell\left(t^{n+1}\right) z_2^\ast , & \text{for } n<0,
\end{cases}
\end{gather}
for all $n\in\bbZ$.
\end{proof}

\subsection{2-cocycle deformations}\label{subsec:deformation}
We construct a $1$-parameter family of noncommutative Hopf fibrations
via $2$-cocyle deformations. We refer the reader to~\cite{Aschieri} for a
general cocycle deformation framework and also to~\cite{Brain,Brain2, LvS}
for the more specific case of toric deformations, which is sufficient for our present work.
Let $K = \mathcal{O}\big(\bbT^2\big)$ denote the $\ast$-Hopf algebra of
functions on the algebraic $2$-torus. As a vector space,
$K$ is spanned by the basis
\begin{gather*}
\big\{t_{\mathbf{m}}\colon \mathbf{m} = (m_1,m_2) \in \bbZ^2\big\} ,
\end{gather*}
on which the product, unit and involution are defined by
\begin{gather*}
t_{\mathbf{m}} t_{\mathbf{m}^\prime} = t_{\mathbf{m} + \mathbf{m}^\prime} ,\qquad
\1 = t_{\mathbf{0}} ,\qquad t_{\mathbf{m}}^\ast = t_{-\mathbf{m}} .
\end{gather*}
The coproduct, counit and antipode are given by
\begin{gather*}
\Delta(t_{\mathbf{m}}) = t_{\mathbf{m}}\otimes t_{\mathbf{m}} ,\qquad
\epsilon(t_{\mathbf{m}}) =1 ,\qquad
S(t_{\mathbf{m}}) = t_{\mathbf{m}}^\ast .
\end{gather*}
We endow $A$ with the structure of a left $K$-comodule
$\ast$-algebra by defining the left $K$-coaction $\rho \colon A\to K \otimes A$
on the generators as
\begin{subequations}\label{eqn:Kcoaction}
\begin{gather}
\rho(z_1) = t_{(1,0)}\otimes z_1 ,\qquad
\rho(z_2) = t_{(0,1)} \otimes z_2 .
\end{gather}
The compatibility condition $\rho\circ \ast = (\ast\otimes\ast)\circ\rho$
between the coaction and $\ast$-involution implies
\begin{gather}
\rho(z_1^\ast) = t_{(-1,0)}\otimes z_1^\ast ,\qquad
\rho(z_2^\ast) = t_{(0,-1)} \otimes z_2^\ast .
\end{gather}
\end{subequations}
Recalling the right $H$-coaction from \eqref{eqn:Hcoaction},
we observe that $(A,\rho,\delta)$ is a $(K,H)$-bicomodule $\ast$-algebra,
i.e., the diagram
\begin{gather}\label{eqn:KHcompatibility}\begin{split}
& \xymatrix@C=4em{
\ar[d]_-{\rho}A \ar[r]^-{\delta} & A\otimes H\ar[d]^-{\rho\otimes \id}\\
K\otimes A \ar[r]_-{\id\otimes\delta} & K\otimes A\otimes H}\end{split}
\end{gather}
commutes and both $\delta$ and $\rho$ are algebra maps.

Let us recall from, e.g., \cite[Section~2.3]{Majid} that a {\em $2$-cocycle} on a Hopf algebra $K$ is
a convolution-invertible linear map $\sigma\colon K\otimes K\to\bbC$
that is unital, i.e., $\sigma(a\otimes \1) = \epsilon(a) = \sigma(\1\otimes a)$
for all $a\in K$, and that satisfies the cocycle condition
\begin{gather}\label{eqn:cocyclecondition}
\sigma(b\sw{1} \otimes c\sw{1}) \sigma(a\otimes b\sw{2} c\sw{2})
= \sigma(a\sw{1} \otimes b\sw{1}) \sigma(a\sw{2} b\sw{2}\otimes c) ,
\end{gather}
for all $a,b,c\in K$. Every $2$-cocycle $\sigma$ on $K$ defines
a deformation of $K$ into a new Hopf algebra~$K_\sigma$
as well as a deformation of the $(K,H)$-bicomodule algebra
$(A,\rho,\delta)$ into a deformed $(K_\sigma,H)$-bicomodule algebra
$(A_\sigma,\rho,\delta)$, see, e.g., \cite[Proposition~2.27]{Aschieri}.
Our focus will be on the family of $2$-cocycles defined by
\begin{subequations}\label{eqn:cocycle}
\begin{gather}
\sigma_\theta (t_{\mathbf{m}} \otimes t_{\mathbf{m}^\prime}) =
\exp\big(\pi {\rm i} \mathbf{m}^\mathrm{T}\Theta \mathbf{m}^\prime\big)
= \exp\left(\pi {\rm i} \sum_{j,k=1}^2 m_j \Theta^{jk} m_k^\prime \right) ,
\end{gather}
where
\begin{gather}
\Theta = \begin{pmatrix}
0 & \theta\\
-\theta & 0
\end{pmatrix} ,\qquad \theta\in\bbR .
\end{gather}
\end{subequations}
(We note in passing that \eqref{eqn:cocycle} may also be interpreted as a $U(1)$-valued
group $2$-cocycle $\widetilde{\sigma}_{\theta}\colon \bbZ^2\times \bbZ^2 \to U(1)$, $(\mathbf{m},\mathbf{m}^\prime)
\mapsto \sigma_\theta (t_{\mathbf{m}} \otimes t_{\mathbf{m}^\prime}) $
on the Pontryagin dual $\big(\bbT^2\big)^\ast = \bbZ^2$.)
As $K = \OO\big(\bbT^2\big)$ is commutative and cocommutative it follows that
$K_{\theta} = K$ as Hopf algebras. The deformed $(K,H)$-bicomodule $\ast$-algebra
$(A_{\theta},\rho,\delta)$ is given as follows. As a $(K,H)$-bicomodule,
we have that $(A_{\theta},\rho,\delta)=(A,\rho,\delta)$, i.e., the left $K$-coaction
and right $H$-coaction remain undeformed. The product is deformed to the $\star_\theta$-product
defined by
\begin{gather}\label{eqn:starproduct}
a \star_\theta a^\prime := \sigma_\theta\big(a\sw{-1}\otimes a^\prime\sw{-1}\big) a\sw 0 a^\prime \sw0 .
\end{gather}
Note that the cocycle condition~\eqref{eqn:cocyclecondition} ensures associativity of
the $\star_\theta$-product.
The unit $\1_\theta = \1$ and involution
$a^{\ast_\theta} = a^\ast$ remain undeformed. Using~\eqref{eqn:Kcoaction},
\eqref{eqn:cocycle} and~\eqref{eqn:starproduct},
one finds the following commutation relations for the generators
\begin{alignat}{3}
 &z_1\star_\theta z_1^\ast = z_1^\ast \star_\theta z_1 ,\qquad && z_2\star_\theta z_2^\ast = z_2^\ast \star_\theta z_2 , &\nonumber\\
& z_1\star_\theta z_2 = q z_2\star_\theta z_1 ,\qquad &&z_1\star_\theta z_2^\ast = q^{-1} z_2^\ast \star_\theta z_1 , &\label{eqn:ConnesLandirelations}
\end{alignat}
where $q= {\rm e}^{2\pi {\rm i} \theta}$. A further calculation shows that
\begin{gather}\label{eqn:starsphererelation}
z_1^\ast\star_\theta z_1 + z_2^\ast \star_\theta z_2 = z_1^\ast z_1 + z_2^\ast z_2 = \1 ,
\end{gather}
i.e., $A_\theta$ describes the Connes--Landi $3$-sphere $\bbS^3_\theta$, see, e.g., \cite{CDV,CL,LvS}.
\begin{lem}\label{lem:deformedcoinvariants}
The $\ast$-subalgebra $A_\theta^{\coH}\subseteq A_\theta$ of $H$-coinvariants
does not depend on the deformation parameter $\theta$,
i.e., $A_\theta^{\coH} = A^{\coH} =B \cong \mathcal{O}\big(\bbS^2\big)$
is the commutative $2$-sphere.
\end{lem}
\begin{proof}Recall from the proof of Lemma \ref{lem:2sphere}
that the undeformed $\ast$-algebra $A^{\coH}$ of coinvariants is generated by
the elements $z := 2 z_1 z_2^\ast$ and $x := z_1^\ast z_1 - z_2^\ast z_2$.
Using~\eqref{eqn:Kcoaction}, we find that the left $K$-coaction on these elements is
\begin{gather}\label{eqn:BcoactionK}
\rho(z) = t_{(1,-1)}\otimes z ,\qquad \rho(z^\ast) = t_{(-1,1)}\otimes z^\ast ,\qquad
\rho(x) = \1\otimes x ,
\end{gather}
which implies via \eqref{eqn:starproduct} and \eqref{eqn:cocycle} that their $\star_\theta$-products coincide with their undeformed pro\-ducts.
\end{proof}

\begin{propo}\label{propo:HGnoncommutative}
The right $H$-comodule $\ast$-algebra $(A_\theta,\delta)$ described above is a principal co\-module algebra.
\end{propo}
\begin{proof}
This follows directly from Proposition \ref{propo:HGcommutative} and \cite[Corollary 3.19]{Aschieri}.
Explicitly, a strong connection can be defined iteratively by
\begin{gather*}
\ell (\1) = \1\otimes \1 , \qquad
\ell (t^n) = \begin{cases}
 z_1^\ast \star_\theta\ell\left(t^{n-1}\right)\star_\theta z_1 + z_2^\ast\star_\theta \ell\left(t^{n-1}\right)\star_\theta z_2,& \text{for } n>0 ,\\
z_1\star_\theta \ell\left(t^{n+1}\right)\star_\theta z^\ast_1 + z_2\star_\theta \ell\left(t^{n+1}\right)\star_\theta z_2^\ast ,& \text{for } n<0 ,
\end{cases}
\end{gather*}
for all $n\in\bbZ$. (Compare this expression with~\eqref{str.con.Hopf}.)
Alternatively, this proposition follows also from \cite[Lemma~3.19]{Brain2}.
\end{proof}

\section{Associated modules}\label{sec:associatedmodules}
\subsection{Modules associated to principal comodule algebras}
Let $H$ be a Hopf algebra.
Given any right $H$-comodule $(M,\delta \colon M\to M\otimes H)$
and left $H$-comodule $(V,\rho \colon V \to H\otimes V)$, the
{\em cotensor product} $M\Box_HV$ is defined as the equalizer
\begin{subequations}\label{eqn:cotensorproduct}
\begin{gather}\label{cotensor}
\xymatrix{
M\square_H V \ar[r] & M\otimes V \ar@<0.5ex>[rr]^-{\delta \otimes \id}\ar@<-0.5ex>[rr]_-{\id \otimes\rho} &&
M\otimes H\otimes V
}
\end{gather}
in the category of vector spaces. Explicitly,
\begin{gather}
M \square_H V = \bigg\{\sum_j m_j\otimes v_j \in M\otimes V
\colon \sum_j \delta (m_j) \otimes v_j = \sum_j m_j\otimes \rho (v_j) \bigg\} .
\end{gather}
\end{subequations}
If $M$ is a $(K,H)$-bicomodule and $V$ is an $(H,L)$-bicomodule
for some other Hopf algebras $K$ and $L$, then the cotensor product
$M \square_H V$ is a $(K,L)$-bicomodule.
\begin{defi}\label{def.associated}
Let $(A,\delta)$ be a principal $H$-comodule algebra
with coinvariant subalgebra $B=A^{\coH}$ and let $(V,\rho)$ be a left $H$-comodule.
Since $\delta$ is a left $B$-module homomorphism, the cotensor product
$E_A(V) := A\Box_H V$ is a left $B$-module via
\begin{gather}\label{eqn:leftBaction}
B\otimes E_A(V)\longrightarrow E_A(V) ,\qquad b\otimes (a\otimes v) \longmapsto (b\,a) \otimes v .
\end{gather}
The corresponding functor
\begin{gather}\label{eqn:associatedmodules}
E_A\colon \ \lcom H \longrightarrow \lmod B
\end{gather}
is called the {\em associated module functor} for the principal comodule algebra $A$.
\end{defi}

As explained in \cite{BrzHaj:Che}, if $V$ is finite-dimensional, then $E_A(V)$ is a finitely generated projective
left $B$-module. In other words, \eqref{eqn:associatedmodules} restricts to a functor
\begin{gather*}
E_A\colon \ \lcomf H \longrightarrow \proj B .
\end{gather*}
An idempotent for $E_A(V)$ can be explicitly constructed from a~strong connection $\ell$ on~$A$ and a~basis of~$V$, see \cite{BrzHaj:Che} and \cite{BohBrz:rel}.

\subsection{Modules associated to Hopf fibrations}\label{subsec:associated}
The aim of this section is to compare the associated module functor $E_{A_\theta}$
for the deformed Hopf fibration $(A_\theta,\delta)$ (cf.\ Proposition~\ref{propo:HGnoncommutative})
with the functor $E_A$ for the classical Hopf fibration $(A,\delta)$
(cf.\ Proposition~\ref{propo:HGcommutative}).
As vector spaces,
\begin{gather*}
E_{A_\theta}(V)=A_\theta \square_H V = A\square_H V= E_A(V) .
\end{gather*}
However, the left $B$-actions on these spaces are different. In view of~\eqref{eqn:leftBaction},
the left $B$-module structure of $E_A(V)$ comes from the commutative multiplication
in $A$, restricted to $B\otimes A$, while the left $B$-module structure of
$E_{A_\theta}(V)$ uses the noncommutative multiplication by $\star_\theta$.
We shall prove below that the two functors $E_{A_\theta}$ and $E_A$
are naturally isomorphic. This means that the theory of associated modules for the deformed
Hopf fibration $(A_\theta,\delta)$ is equivalent to that for the classical Hopf fibration.
Loosely speaking, it could be said that associated modules do not
depend on the deformation parameter $\theta$.

In order to construct this natural isomorphism we have to analyze
the two different left $B$-module structures in more detail. For this, it is convenient to
decompose the underlying $(K,H)$-bicomodule $(A_\theta ,\rho,\delta) = (A,\rho,\delta)$
into irreducible representations. We define the homogeneous
$(K,H)$-bicomodules
\begin{gather*}
A^{(\mathbf{m},n)} := \big\{ a\in A \colon \rho(a) = t_{\mathbf{m}}\otimes a,\,
 \delta(a) = a\otimes t^n\big\} \subseteq A ,
\end{gather*}
for all $\mathbf{m}=(m_1,m_2)\in \bbZ^2$ and $n\in\bbZ$.
\begin{lem}
There exists a decomposition of the $(K,H)$-bicomodule
$(A_\theta ,\rho,\delta) = (A,\rho,\delta)$ as
\begin{gather}\label{eqn:Adecomposition}
A = \bigoplus_{m,n\in\bbZ} A^{((m+n,-m),n)} .
\end{gather}
The $\star_\theta$-product of homogeneous elements
$a\in A^{((m+n,-m),n)} $ and $a^\prime \in A^{((m^\prime+n^\prime,-m^\prime),n^\prime)} $
is
\begin{gather}\label{eqn:starproductdecomposition}
a\star_\theta a^\prime = {\rm e}^{\pi {\rm i} \theta (m n^\prime - n m^\prime)} a a^\prime \in
A^{((m+m^\prime +n + n^\prime ,-m - m^\prime ),n + n^\prime)} .
\end{gather}
\end{lem}
\begin{proof}Because $A$ is the $(K,H)$-bicomodule underlying a
finitely presented $(K,H)$-bicomodule algebra, there exists a decomposition
$A= \bigoplus_{\mathbf{m}\in\bbZ^2 ,n\in\bbZ} A^{(\mathbf{m},n)}$,
cf.\ \cite[Lemma~A.3]{Barnes}. As the generators
are homogeneous elements
\begin{gather*}
z_1\in A^{((1,0),1)}, \qquad z_2\in A^{((0,1),1)}, \qquad z_1^\ast \in A^{((-1,0),-1)},\qquad z_2^\ast\in A^{((0,-1),-1)} ,
\end{gather*}
we observe that the non-vanishing components are as in~\eqref{eqn:Adecomposition}.
The formula for the $\star_\theta$-product on homogeneous elements
follows directly from~\eqref{eqn:starproduct} and~\eqref{eqn:cocycle}.
\end{proof}

It follows from \eqref{eqn:Adecomposition} that
\begin{gather*}
B = A^{\coH} = \bigoplus_{m\in\bbZ} A^{((m,-m),0)} \subseteq A
\end{gather*}
for the $\ast$-subalgebra of $H$-coinvariants.

Given any left $H$-comodule $(V,\rho)$, the vector space underlying
the associated left $B$-module admits a decomposition
\begin{gather*}
A \square_H V = \bigoplus_{m,n\in \bbZ} A^{((m+n,-m),n)}\otimes V^{n} ,
\end{gather*}
where
\begin{gather*}
V^n := \big\{ v \in V \colon \rho(v) = t^n \otimes v \big\} .
\end{gather*}
By \eqref{eqn:starproductdecomposition},
the deformed left $B$-action of homogeneous elements $b\in A^{((m,-m),0)}\subseteq B$
and $a\otimes v \in A^{((m^\prime + n^\prime ,-m^\prime ),n^\prime ) }\otimes V^{n^\prime} \subseteq A\square_H V$
is given by
\begin{gather*}
b\star_\theta (a\otimes v) = (b\star_\theta a)\otimes v = {\rm e}^{\pi {\rm i} \theta m n^\prime} (b\,a)\otimes v .
\end{gather*}
\begin{propo}\label{propo:moduleequivalence}
For every left $H$-comodule $(V,\rho)$, define a linear isomorphism
\begin{subequations}\label{eqn:Lcomponents}
\begin{gather}
L_{V}^{}\colon \ A \square_H V \longrightarrow A_\theta \square_H V
\end{gather}
by setting
\begin{gather}
L_{V}^{}\big(a\otimes v\big) = {\rm e}^{\pi {\rm i} \theta mn} a\otimes v ,
\end{gather}
\end{subequations}
for all homogeneous elements $a\otimes v\in A^{((m+n,-m),n)} \otimes V^n$. This linear isomorphism
is a left $B$-module isomorphism for $E_A(V)$ and $E_{A_\theta}(V)$.
Moreover, the components~\eqref{eqn:Lcomponents} define a natural
isomorphism $L \colon E_{A}\Longrightarrow E_{A_\theta}$
between the associated module functor for the classical Hopf fibration
$\bbS^3\to \bbS^2$ and the one for the deformed Hopf fibration $(A_\theta,\delta)$.
\end{propo}
\begin{proof}
For homogeneous elements $b\in A^{((m,-m),0)}\subseteq B$
and $a\otimes v \in A^{((m^\prime + n^\prime ,-m^\prime ),n^\prime )}\otimes V^{n^\prime} \subseteq A\square_H V$,
we obtain
\begin{gather*}
L_{V}^{}\big(b (a\otimes v)\big) = {\rm e}^{\pi {\rm i} \theta (m+m^\prime)n^\prime} (b a)\otimes v ,
\end{gather*}
where $(b a)\otimes v\in A^{((m+m^\prime + n^\prime ,-m-m^\prime ),n^\prime) }\otimes V^{n^\prime} $,
and
\begin{gather*}
b\star_\theta L_{V}^{}\big( a \otimes v\big) = {\rm e}^{\pi{\rm i} \theta m^\prime n^\prime} (b\star_\theta a)\otimes v =
 {\rm e}^{\pi {\rm i} \theta m^\prime n^\prime} {\rm e}^{\pi {\rm i} \theta m n^\prime} (b a)\otimes v .
\end{gather*}
These two terms coincide, hence $L_{V}^{}$ is a left $B$-module
isomorphism. Naturality is a straightforward check.
\end{proof}
\begin{remm}\label{rem:nococycle}
We would like to emphasize that Proposition~\ref{propo:moduleequivalence}
is {\em not} a consequence of the general theory of $2$-cocycle deformations
from~\cite{Brain,Brain2} and~\cite{Aschieri}.
In such a setting, it is convenient to observe that
the $K$-coaction on $A$ induces to associated modules,
i.e., $E_A \colon \lcom H \to {}_{B}^{K}\mathscr{M}$
can be regarded as a functor to the category
of left $K$-comodule left $B$-modules.
Consequently, the deformed associated module functor
$E_{A_\theta} \colon \lcom H \to {}_{B_\theta}^{K_\theta}\mathscr{M} $
can be regarded as a functor to the category
of left $K_\theta$-comodule left $B_\theta$-modules,
where $K_\theta$ and $B_\theta$ are the $2$-cocycle deformations of
$K$ and $B$. (In our specific example of interest, we have that $K_\theta=K$
and $B_\theta=B$ (cf.\ Lemma~\ref{lem:deformedcoinvariants}),
but we shall keep these labels to make the discussion below more transparent.)
Using the $2$-cocycle deformation functor
from \cite[Proposition~2.4]{Brain2} or \cite[Proposition~2.25]{Aschieri},
which we shall denote by $\Sigma_\theta$, one obtains a commutative diagram
\begin{gather*}
\xymatrix@C=4em{
\ar[rd]_-{E_{A_\theta}} \lcom H \ar[r]^-{E_A} & {}_{B}^{K}\mathscr{M}\ar[d]^-{\Sigma_\theta}\\
 & {}_{B_\theta}^{K_\theta}\mathscr{M}
}
\end{gather*}
relating the associated module functor $E_{A_\theta}$ for the deformed
Hopf fibration to the associated module functor $E_A$ for the classical one.
In words, each deformed associated module $E_{A_\theta}(V)$ can be determined
by applying the deformation functor to the classical associated module $E_A(V)$.
For general $2$-cocycle deformations that is all one can say.

For our special example given by the deformed Hopf fibration $A_\theta$,
we have that $K_\theta=K$ and $B_\theta =B$ are undeformed,
hence the deformed and the undeformed associated
module functors $E_A , E_{A_\theta} \colon \lcom H \to {}_{B}^{K}\mathscr{M}$
have the same target category. Proposition~\ref{propo:moduleequivalence}
proves that these two functors are already `the same'
(in the sense of naturally isomorphic) even if we do not use the deformation
functor $\Sigma_\theta$. Hence, our natural isomorphism in
Proposition \ref{propo:moduleequivalence} is more special and
stronger than the results from the general theory of $2$-cocycle
deformations from~\cite{Brain,Brain2} and~\cite{Aschieri}.
\end{remm}

\section{The Atiyah sequence and connections}\label{sec:connections}
\subsection{Differential geometry of principal comodule algebras}
Let us start with a brief review of some relevant concepts from
noncommutative differential calculi, see, e.g., \cite{Landi} and \cite{BegMaj:Rie} for more details.
\begin{defi}\label{def.diff.calc}
A {\em $($first-order$)$ differential calculus} on an algebra $A$
is a pair $\big(\Omega^1(A),\dd\big)$ consisting of an $A$-bimodule $\Omega^1(A)$
and a linear map $\dd\colon A\to \Omega^1(A)$, such that
\begin{itemize}\itemsep=0pt
\item[(i)] $\dd(a a^\prime) = a \dd(a^\prime) + \dd(a) a^\prime$,
for all $a,a^\prime\in A$,
\item[(ii)] $\Omega^1(A) = A \dd(A) = \big\{\sum_j a_j \dd(a_j^\prime) \colon a_j,a_j^\prime\in A\big\}$.
\end{itemize}
We say that $\big(\Omega^1(A),\dd\big)$ is {\em connected} if
$\dd(a) = 0$ if and only if $a\in k\1 \subseteq A$.
\end{defi}

Every differential calculus is a quotient
of the {\em universal calculus} $\big(\Gamma^1(A),\du\big)$.
Recall that the $A$-bimodule of {\em universal $1$-forms}
$\Gamma^1(A) := \ker\mu_A \subseteq A\otimes A$ is the kernel
of the multiplication map and that the {\em universal differential}
$\du \colon A\to \Gamma^1(A)$ is defined as $\du(a) = \1\otimes a - a\otimes \1$,
for all $a\in A$. We find that $\big(\Gamma^1(A),\du\big)$ is a connected
differential calculus. Given any $A$-subbimodule $N\subseteq \Gamma^1(A)$,
the $A$-bimodule $\Omega^1(A) = \Gamma^1(A)/N$ and
\begin{gather*}
\xymatrix{
A \ar[rr]^-{\dd} \ar[dr]_-{\du} && \Omega^1(A)\\
& \Gamma^1(A)\ar@{->>}[ur] &
}
\end{gather*}
defines a differential calculus $\big(\Omega^1(A),\dd\big)$. Vice versa,
every differential calculus is of this form, see, e.g., \cite[Proposition~6.1]{Landi} for a proof.

The case where $H$ is a Hopf algebra was studied in detail in \cite{Wor:dif}.
Any right ideal $Q\subseteq H^+$ of the augmentation ideal $H^+:=\ker \epsilon$
that is invariant under the right adjoint $H$-coaction $\ad\colon H^+\to H^+\otimes H$, $h\mapsto h\sw 2 \otimes S(h \sw 1 ) h\sw 3$,
i.e., $\ad(Q) \subseteq Q\otimes H$, induces a {\em bicovariant differential
calculus} on $H$. The corresponding $H$-subbimodule $N_H\subseteq \Gamma^1(H)$
is generated by the image of $Q\subseteq H^+$ under the linear map
$\kappa := (S\otimes\id)\circ\Delta\colon H^+ \to \Gamma^1(H)$.
The quotient right $H$-comodule
\begin{gather*}
\mathfrak{h}^\vee_Q := H^+ / Q
\end{gather*}
plays the role of the dual of the quantum Lie algebra of $H$ relative to $Q$.
By construction, there exists a map $\kappa_Q$ that fits into the commutative diagram
\begin{gather*}
\xymatrix{
0 \ar[r] & Q \ar[r]\ar[d]_-{\kappa} & H^+ \ar[r]\ar[d]_-{\kappa} & \mathfrak{h}^\vee_Q \ar[r]\ar@{-->}[d]^-{\kappa_Q} &0 \,\\
0 \ar[r] & N_H \ar[r] & \Gamma^1(H) \ar[r] & \Omega^1(H) \ar[r] &0 ,
}
\end{gather*}
whose rows are short exact sequences.
The resulting $H$-bimodule of $1$-forms $\Omega^1(H) = \Gamma^1(H)/N_H$
is freely generated by $\kappa_Q(\chi_i)\in \Omega^1(H)$, where $\{\chi_i\}$
is a basis of $\mathfrak{h}^\vee_Q$.

For a right $H$-comodule algebra $(A,\delta)$, it is natural
to demand that the $A$-subbimodule $N\subseteq \Gamma^1(A)$
is invariant under the tensor product right $H$-coaction, i.e., $\delta(N)\subseteq N\otimes H$. More explicitly,
\begin{gather*}
\sum_{j} (a_j)\sw 0 \otimes (a_j^\prime)\sw 0 \otimes (a_j)\sw 1 (a_j^\prime)\sw 1 \in N\otimes H ,
\end{gather*}
for all $\sum_j a_j\otimes a_j^\prime \in N$. The corresponding
differential calculus $\big(\Omega^1(A),\dd\big)$ then satisfies the
property that $\Omega^1(A)$ is a right $H$-comodule $A$-bimodule
and that $\dd\colon A\to \Omega^1(A)$ is a right $H$-comodule morphism.
Further, given any $\ad$-invariant right ideal $Q\subseteq H^+$,
we may require that~$Q$ and~$N$ are compatible in the sense that
$\ver(N) \subseteq A\otimes Q$, where the {\em $($universal$)$ vertical
lift} is defined as the linear map
\begin{align}
\nn \ver := (\mu_A\otimes \id)\circ (\id\otimes \delta) \colon \ \Gamma^1(A) &\longrightarrow A\otimes H^+, \\
\sum_j a_j\otimes a^\prime_j &\longmapsto \sum_j a_j (a^\prime_j)\sw 0\otimes (a^\prime_j)\sw 1 .\label{lift}
\end{align}
Observe that in this case, the (universal) vertical lift descends to a linear map
$\overline{\ver} \colon \Omega^1(A) \to A\otimes \mathfrak{h}^\vee_Q$ defined by
the diagram
\begin{gather*}
\xymatrix{
\Gamma^1(A)\ar[r]\ar[d]_-{\ver} & \Omega^1(A) \ar[r]\ar@{-->}[d]^-{\overline{\ver}} &0\\
A\otimes H^+ \ar[r] & A\otimes \mathfrak{h}^\vee_Q \ar[r] &0.
}
\end{gather*}
If moreover there is an equality $\ver(N) = A\otimes Q$, we say that the two differential
calculi $\big(\Omega^1(A),\dd\big)$ and $\big(\Omega^1(H),\dd\big)$ are {\em concordant}.
This definition is motivated by the result in~\cite{Haj:not} that the non-universal
Atiyah sequence (see~\eqref{general.Atiyah} below) associated to a principal
$H$-comodule algebra~$(A,\delta)$ is short exact if and only if the
differential calculi on~$A$ and~$H$ are concordant.
Loosely speaking, this means that in the case of
concordant differential calculi one has an identification
between the vertical vector fields on~$A$ and
the quantum Lie algebra of the structure Hopf algebra, which is
in analogy to the fundamental vector field construction from
classical differential geometry.

It is well-known, see, e.g., \cite[Part~VII, Proposition~6.6]{BJM},
that a right $H$-comodule algebra $(A,\delta)$ is a Hopf--Galois extension
of the coinvariant subalgebra $B=A^{\coH}$ if and only if the {\em $($universal$)$ Atiyah sequence}
\begin{gather}\label{universal.Atiyah}
\xymatrix{
0 \ar[r] & A \Gamma^1(B) A \ar[r]& \Gamma^1(A) \ar[r]^-{\mathrm{ver}} & A\otimes H^{+} \ar[r] & 0
}
\end{gather}
of right $H$-comodule left $A$-modules is short exact. In particular,
for every principal $H$-comodule algebra
$(A,\delta)$, the sequence~\eqref{universal.Atiyah} is short exact.
The right $H$-comodule $A$-bimodule
\begin{gather*}
A\Gamma^1(B) A := \left\{\sum_j a_j b_j\otimes b^\prime_j a^\prime_j \colon
a_j,a^\prime_j\in A ,\, b_j,b^\prime_j \in B
\text{ with } \sum_j a_j b_j b^\prime_j a^\prime_j =0\right\}
\end{gather*}
is called the {\em module of universal horizontal $1$-forms}.
Let us also recall that the right $H$-coactions on both $A\Gamma^1(B) A$ and $\Gamma^1(A)$
are induced by the tensor product coaction on $A\otimes A$. As well,
the right $H$-coaction on $A\otimes H^+$ is the tensor product coaction
with $H^+$ endowed with the right adjoint $H$-coaction $\ad\colon h\mapsto h\sw 2\otimes S(h\sw 1) h\sw 3$.
\begin{defi}A {\em connection} with respect to the universal differential calculus
on a principal $H$-comodule algebra $(A,\delta)$ is a
right $H$-comodule left $A$-module splitting $\s\colon A\otimes H^+ \to \Gamma^1(A)$
of \eqref{universal.Atiyah}, i.e., $\ver\circ \s = \id$.
Every connection $\s$ is fully determined by its {\em connection form}
\begin{gather*}
\omega\colon \ H^+ \longrightarrow \Gamma^1(A), \qquad h\longmapsto \s(\1_A\otimes h) .
\end{gather*}
\end{defi}

Associated to a connection $\s\colon A\otimes H^+ \to \Gamma^1(A)$,
or equivalently to its connection form $\omega \colon H^+ \to \Gamma^1(A)$,
is its {\em covariant derivative}
\begin{gather*}
D := \left(\id - \s\circ \ver\right)\circ \du\colon \ A \longrightarrow A\Gamma^1(B) A,\qquad
a\longmapsto \du (a) - a\sw 0 \omega\big(a\sw 1- \epsilon(a\sw 1)\1_H\big) .
\end{gather*}
A connection is called {\em strong} provided that $D(A)\subseteq \Gamma^1(B)\,A$.
Strong connections are in one-to-one correspondence with maps $\ell\colon H\to A\otimes A$
satisfying the conditions~\eqref{str.con} in Definition~\ref{def:princ.com.alg},
see, e.g., \cite{DabGro:str} or~\cite{BrzHaj:Gal}. A covariant derivative $D$
of a strong connection is a connection on the left $B$-module $A$, in the sense of
\cite{CunQui:alg}. Every strong connection defines a connection, for the universal calculus, on the associated left $B$-module $E_A(V) = A\Box_H V$ via
\begin{gather*}
\nabla \colon \ E_A(V) \longrightarrow \Gamma^1(B)\otimes_B E_A(V) ,\qquad
a\otimes v \longmapsto D(a)\otimes v.
\end{gather*}

It was shown in \cite{Haj:not} that for
concordant differential calculi $\big(\Omega^1(A),\dd\big)$ and $\big(\Omega^1(H),\dd\big)$,
i.e., $\ver(N)=A\otimes Q$, exactness of the universal Atiyah sequence
\eqref{universal.Atiyah} implies exactness of the induced sequence
\begin{gather}\label{general.Atiyah}
 \xymatrix{
0 \ar[r] & \Omega^1(A)_{\mathrm{hor}} \ar[r] & \Omega^1(A) \ar[r]^-{\overline{\mathrm{ver}}} & A\otimes \mathfrak{h}^\vee_Q \ar[r] & 0}
\end{gather}
of right $H$-comodule left $A$-modules.
The {\em module of horizontal $1$-forms} is defined by
\begin{gather*}
 \Omega^1(A)_{\mathrm{hor}} :=
 \frac{A\Gamma^1(B)A}{A\Gamma^1(B)A \cap N} =
 A\Omega^1(B)A ,
\end{gather*}
where $\big(\Omega^1(B),\dd\big)$ is the differential calculus on $B$ that is
determined by the $B$-subbimodule $N_B := (\ker \mu_B)\cap N\subseteq \Gamma^1(B)$.
Equivalently, this differential calculus may by
obtained from $\big(\Omega^1(A),\dd\big)$ as a restriction, i.e.,
$\Omega^1(B) = B \dd(B)\subseteq \Omega^1(A)$, see, e.g.,
\cite[Chapter~5]{BegMaj:Rie}.
The short exact sequence~\eqref{general.Atiyah} will be referred to as the {\em Atiyah sequence}
for the concordant differential calculi $\big(\Omega^1(A),\dd\big)$ and
$\big(\Omega^1(H),\dd\big)$.
\begin{defi}\label{def:connectiongeneralcalculus}
A {\em connection} with respect to concordant differential calculi
$\big(\Omega^1(A),\dd\big)$ and $(\Omega^1(H),\dd)$ on a principal $H$-comodule algebra
$(A,\delta)$ is a right $H$-comodule left $A$-module splitting $\overline{\s}
\colon A\otimes \mathfrak{h}^\vee_Q \to \Omega^1(A)$ of the Atiyah sequence~\eqref{general.Atiyah}, i.e.,
$\overline{\ver}\circ \overline{\s} = \id$.
Every connection $\overline{\s}$ is fully determined by its {\em connection form}
\begin{gather*}
\overline{\omega} \colon \ \mathfrak{h}^\vee_Q \longrightarrow \Omega^1(A), \qquad h\longmapsto
\overline{\s}(\1_A\otimes h) .
\end{gather*}
\end{defi}

Similarly to the case of the universal calculus above,
associated to a connection $\overline{\s} \colon A\otimes \mathfrak{h}^\vee_Q \to \Omega^1(A)$,
or equivalently to its connection form $\overline{\omega} \colon \mathfrak{h}^\vee_Q \to \Omega^1(A)$,
is its {\em covariant derivative}
\begin{gather}\label{cov.derivgeneral}
\overline{D} := \left(\id - \overline{\s}\circ \overline{\ver}\right)\circ \dd\colon \ A \longrightarrow
\Omega^1_{\mathrm{hor}}(A), \qquad
a\longmapsto \dd (a) - a\sw 0 \overline{\omega}\big(a\sw 1- \epsilon(a\sw 1)\1_H \big) .
\end{gather}
A connection is said to be {\em strong} provided that $\overline{D}(A)\subseteq \Omega^1(B) A$.
Every strong connection defines a connection for the differential calculus $\big(\Omega^1(B),\dd\big)$
on the associated left $B$-module $E_A(V) = A\Box_H V$ via
\begin{gather}\label{eqn:Nablageneral}
\overline{\nabla}\colon \ E_A(V) \longrightarrow \Omega^1(B)\otimes_B E_A(V) , \qquad
a\otimes v \longmapsto \overline{D}(a)\otimes v .
\end{gather}

\subsection{The universal Atiyah sequence for Hopf fibrations}
Consider the $(K,H)$-bicomodule algebra $(A_\theta,\rho,\delta)$
from Section \ref{subsec:deformation} which describes the deformed Hopf fibration.
As the underlying right $H$-comodule algebra $(A_\theta,\delta)$
is a principal comodule algebra (cf.\ Proposition \ref{propo:HGnoncommutative}),
we obtain from \eqref{universal.Atiyah} the corresponding (universal) Atiyah sequence
\begin{gather}\label{eqn:universalAtiyah}
\xymatrix{
0 \ar[r] & A_\theta \Gamma^1(B) A_\theta \ar[r]& \Gamma^1(A_\theta) \ar[r]^-{\mathrm{ver}_\theta} & A_\theta\otimes H^{+} \ar[r] & 0.}
\end{gather}
Let us emphasize that $B= A_\theta^\coH\subseteq A_\theta$ is the algebra of functions on the classical $2$-sphere
(cf.\ Lemma~\ref{lem:deformedcoinvariants}) and that the vertical lift \eqref{lift}, in the present case, involves the $\star_\theta$-product of $A_\theta$, i.e.,
\begin{gather*}
\mathrm{ver}_\theta \colon \ \Gamma^1(A_\theta) \longrightarrow A_\theta \otimes H^{+}, \qquad
\sum_{j} a_j \otimes a^\prime_j \longmapsto \sum_{j} \big(a_j\star_\theta (a^\prime_j)\sw 0 \big)
\otimes (a^\prime_j)\sw1 .
\end{gather*}
Note that \eqref{eqn:universalAtiyah} is a short exact sequence of $(K,H)$-bicomodule left $A_\theta$-modules.
Setting the deformation parameter $\theta=0$, we obtain the (universal) Atiyah sequence
of the classical Hopf fibration $\bbS^3\to \bbS^2$
\begin{gather}\label{eqn:universalAtiyahclassical}
\xymatrix{0 \ar[r] & A \Gamma^1(B) A \ar[r]& \Gamma^1(A) \ar[r]^-{\mathrm{ver}} & A\otimes H^{+} \ar[r] & 0.}
\end{gather}
This is a short exact sequence of $(K,H)$-bicomodule left $A$-modules. Utilizing
the $2$-cocycle deformation functor from \cite[Proposition~2.4]{Brain2} or
\cite[Proposition~2.25(ii)]{Aschieri}, we obtain a short exact sequence of
$(K,H)$-bicomodule left $A_\theta$-modules
\begin{gather}\label{eqn:universalAtiyahdeformed}
\xymatrix{
0 \ar[r] & \big(A \,\Gamma^1(B)\, A\big)_\theta \ar[r]& \Gamma^1(A)^{~}_\theta \ar[r]^-{\mathrm{ver}} & \big(A\otimes H^{+}\big)_\theta \ar[r] & 0.
}
\end{gather}
As a sequence of $(K,H)$-bicomodules, \eqref{eqn:universalAtiyahdeformed}
coincides with \eqref{eqn:universalAtiyahclassical}. The left $A_\theta$-module structures
on the objects in \eqref{eqn:universalAtiyahdeformed} are obtained by the following construction:
Given any $(K,H)$-bicomodule left $A$-module
$M$, the deformed left $A_\theta$-action on the $(K,H)$-bicomodule left $A_\theta$-module $M_\theta$ is
\begin{gather*}
a\star_\theta m := \sigma_{\theta}(a\sw{-1}\otimes m\sw{-1}) a\sw 0 m\sw 0 ,
\end{gather*}
for all $a\in A_\theta $ and $m\in M_\theta$.
\begin{propo}\label{propo:universalAtiyahiso}
Consider the $(K,H)$-bicomodule isomorphism
\begin{gather}\label{eqn:varphimap}
\varphi_\theta\colon \ A_\theta\otimes A_\theta \longrightarrow (A\otimes A)_\theta, \qquad
a\otimes a^\prime \longmapsto \sigma_\theta(a\sw{-1}\otimes a^\prime \sw{-1}) a\sw 0\otimes a^\prime\sw 0 .
\end{gather}
Then
\begin{gather}\label{eqn:Atiyahmap}\begin{split}&
\xymatrix{
0 \ar[r] & \ar[d]_-{\varphi_\theta }A_\theta \Gamma^1(B) A_\theta \ar[r]& \ar[d]_-{\varphi_\theta} \Gamma^1(A_\theta) \ar[r]^-{\mathrm{ver}_\theta} & A_\theta\otimes H^{+} \ar[d]^-{\id}\ar[r] & 0\\
0 \ar[r] & \big(A \Gamma^1(B) A\big)_\theta \ar[r]& \Gamma^1(A)_\theta \ar[r]^-{\mathrm{ver}} & \big(A\otimes H^{+}\big)_\theta \ar[r] & 0
}\end{split}
\end{gather}
is an isomorphism of short exact sequences of $(K,H)$-bicomodule left $A_\theta$-modules.
\end{propo}
\begin{proof}Note that \eqref{eqn:varphimap} relates the deformed and undeformed
product via $\mu_{A_\theta} = \mu_A \circ \varphi_\theta$. As a consequence,
it restricts to the middle vertical arrow in~\eqref{eqn:Atiyahmap}.
The $2$-cocycle property of~$\sigma_\theta$ implies that the middle vertical arrow
is a $(K,H)$-bicomodule left $A_\theta$-module isomorphism. From~\eqref{eqn:BcoactionK}
we find that $\varphi_\theta$ acts as the identity on
$\Gamma^1(B)\subseteq B\otimes B \subseteq A_\theta\otimes A_\theta$. Together with the previous result,
this implies that the left vertical arrow in~\eqref{eqn:Atiyahmap} has the claimed domain
and codomain and that it is a $(K,H)$-bicomodule left $A_\theta$-module isomorphism.
The left square commutes by construction.

The right vertical arrow in \eqref{eqn:Atiyahmap} is a $(K,H)$-bicomodule left $A_\theta$-module isomorphism
because the $K$-coaction on $H^+$ is trivial. We see directly that the right square commutes by
\begin{align*}
\nn \mathrm{ver}\circ \varphi_\theta \bigg(\sum_j a_j\otimes a^\prime_j\bigg)
&=\sum_j \sigma_\theta\big((a_j)\sw{-1}\otimes (a^\prime_j)\sw{-1}\big) (a_{j})\sw 0 (a^\prime_j)\sw{0}_{\sw{0}}
\otimes (a^\prime_j)\sw{0}_{\sw{1}}\\
\nn &=\sum_j \sigma_\theta\big((a_j)\sw{-1}\otimes (a^\prime_j)\sw{0}_{\sw{-1}} \big) (a_{j})\sw 0
 (a^\prime_j)\sw{0}_{\sw{0}}\otimes (a^\prime_j)\sw{1}\\
&=\sum_j a_{j} \star_\theta (a^\prime_j)\sw{0}\otimes (a^\prime_j)\sw{1} = \mathrm{ver}_\theta\bigg(\sum_j a_j\otimes a^\prime_j\bigg) ,
\end{align*}
for all $\sum_j a_j\otimes a^\prime_j \in \Gamma^1(A_\theta)$, where in
the second step we used~\eqref{eqn:KHcompatibility}.
\end{proof}

\subsection{The Atiyah sequence for K\"ahler forms}
For a {\em commutative} algebra $A$, the product map
$\mu_A\colon A\otimes A\to A$ is an algebra homomorphism
when $A\otimes A$ is endowed with the tensor algebra structure
$(a\otimes a^\prime) (\widetilde{a}\otimes \widetilde{a}^\prime) :=
(a \widetilde{a})\otimes(a^\prime \widetilde{a}^\prime)$.
This implies that $\ker \mu_A\subseteq A\otimes A$ is an ideal.
Recall that the module of {\em K\"ahler $1$-forms} on $A$
is defined as the quotient $A$-bimodule
\begin{gather*}
\Omega^1(A) := \frac{\ker\mu_A}{(\ker\mu_A)^2} = \frac{\Gamma^1(A)}{(\ker\mu_A)^2} .
\end{gather*}
The K\"ahler differential $\dd\colon A\to \Omega^1(A)$
is the composition of the universal differential
$\du\colon A\to\Gamma^1(A)$ and the quotient map
$\Gamma^1(A) \twoheadrightarrow \Omega^1(A)$. In other words, $\big(\Omega^1(A),\dd\big)$ is
the first-order differential calculus presented by the quotient
of the universal calculus $\big(\Gamma^1(A),\du\big)$ by the $A$-subbimodule
$N := (\ker \mu_A)^2 \subseteq \Gamma^1(A)$.

For a {\em commutative} Hopf algebra $H$, the K\"ahler differential
calculus on $H$ is bicovariant and it corresponds to the $\ad$-invariant right
ideal $Q = (H^+)^2$, where $H^+ = \ker \epsilon$.
Given further a~{\em commutative} principal $H$-comodule algebra
$(A,\delta)$, with coinvariants $B=A^{\coH}$, both
$\Gamma^1(A) = \ker\mu_A\subseteq A\otimes A$ and $A\otimes H^+\subseteq A\otimes H$
are ideals with respect to the tensor algebra structures. It is easily checked
that the vertical lift $\ver\colon \Gamma^1(A)\to A\otimes H^+$
is an algebra homomorphism, hence it maps $(\ker\mu_A)^2$ to
$(A\otimes H^+)^2 = A\otimes (H^+)^2$.
Because the latter map is surjective, the K\"ahler differential calculi
on~$A$ and~$H$ are concordant.
Hence, there is a corresponding short exact Atiyah sequence~\eqref{general.Atiyah}
for K\"ahler forms in which
\begin{gather*}
\Omega^1_\mathrm{hor}(A) := \frac{ A \Gamma^1(B) A}{\big(A \Gamma^1(B) A\big) \cap (\ker\mu_A)^2} = A\Omega^1(B)A ,\qquad \mathfrak{h}^{\vee} := \mathfrak{h}^{\vee}_Q = \frac{H^+}{ (H^+)^2} ,
\end{gather*}
and $\Omega^1(B) = \Gamma^1(B)/N_B$ is determined by $N_B = (\ker \mu_A)^2\cap \Gamma^1(B) \subseteq \Gamma^1(B)$.

It might be worth pointing out that, in general,
$\Omega^1(B)$ defined above by the restriction of the K\"ahler $1$-forms
on $A$ is not necessarily the module of K\"ahler $1$-forms on $B$.
\begin{ex}\label{ex.non.Ka}
Let $k$ be a field of characteristic $p$ and $A = k\big[x,x^{-1}\big]$ the commutative
algebra of Laurent polynomials in $x$. Further, let $H=k\bbZ_p$ be the group Hopf algebra
of the cyclic group of order $p$. The latter may be presented as the commutative
algebra $H = k[\xi]/\big(\1-\xi^p\big)$,
equipped with the coproduct $\Delta(\xi) = \xi \otimes \xi$,
counit $\epsilon(\xi)=1$ and antipode $S(\xi) = \xi^{-1}$. (By $\xi^{-1}\in H$
we mean the element represented by $\xi^{p-1}\in k[\xi]$.)
We endow $A$ with the structure of a right $H$-comodule algebra
by defining $\delta(x) = x\otimes \xi$ and $\delta\big(x^{-1}\big) = x^{-1}\otimes \xi^{-1}$.
The $H$-coinvariant subalgebra
$B = k\big[x^p,x^{-p}\big]\subseteq k\big[x,x^{-1}\big]=A$ is the algebra of
Laurent polynomials in~$x^p$.
We further observe that
$(A,\delta)$ is a principal comodule algebra with strong connection
$\ell\colon H\to A\otimes A$ defined by
\begin{gather*}
\ell(\xi^n) = x^{-n}\otimes x^n ,\qquad\text{for} \quad n=0,\dots,p-1 .
\end{gather*}
Because $A = k\big[x,x^{-1}\big] \cong k[x,y]/(xy-\1)$ admits a finite presentation
by two generators $x$ and $y$, the corresponding module of K\"ahler $1$-forms
is isomorphic to the quotient $\Omega^1(A) \cong (A\,\dd x \oplus A\,\dd y) /\langle y\,\dd x + x\,\dd y\rangle$,
see, e.g., \cite[Section~16.1]{Eis:com}. The generator $\dd y$ can be eliminated
by the relation $\dd y = -y^{2} \dd x$, which implies that $\Omega^1(A) = A\,\dd x$. (Note
that $\dd x^{-1} = -x^{-2} \dd x$ in this calculus.)
Restricting the K\"ahler $1$-forms $\Omega^1(A)$ to $B= k\big[x^p,x^{-p}\big]\subseteq k\big[x,x^{-1}\big] = A$
defines the trivial differential calculus $\Omega^1(B) = 0$. This follows as $k$ is
a field of characteristic~$p$, by hypothesis, and hence $\dd (x^p) = p\, x^{p-1}\,\dd x =0$. On the other hand, the K\"ahler differential calculus on $B$ is non-trivial.\footnote{We
are grateful to Christian Lomp for his comment, which led us to Example~\ref{ex.non.Ka}.}
\end{ex}

In view of the above comment and example, it is useful to observe the following lemma.
\begin{lem}\label{lem.Ka}
Let $A=k[x_1,\dots,x_n]/J_A$ be a finitely generated algebra
and $B\subseteq A$ the subalgebra generated by $X_1,\dots,X_m \in k[x_1,\dots,x_n]$.
If the set $\{ \dd_x (X_1) ,\dots, \dd_x (X_m)\}$ is free in the module underlying
the K\"ahler differential calculus $\big(\Omega^1(k[x_1,\dots,x_n]),\dd_x\big)$ on the algebra $k[x_1,\dots,x_n]$, then the K\"ahler differential calculus $\big(\Omega^1(B),\dd_B\big)$ on $B$ is isomorphic to the restriction $(B\dd_A(B),\dd_A) \subseteq \big(\Omega^1(A),\dd_A\big)$ to $B\subseteq A$
of the K\"ahler differential calculus on $A$.
\end{lem}
\begin{proof}
Recall from \cite[Proposition~16.1]{Eis:com} that $\Omega^1(k[x_1,\dots, x_n])$
and $\Omega^1(k[X_1,\dots, X_m])$ are free modules. By our hypothesis
that $\{ \dd_x (X_1) ,\dots, \dd_x (X_m)\}$ is free in $\Omega^1(k[x_1,\dots,x_n])$,
it follows that the map
\begin{subequations}\label{iso.Ka.pol}
\begin{gather}
\Omega^1(k[X_1,\ldots, X_m]) \longrightarrow \Omega^1(k[x_1,\ldots, x_n]) , \qquad \sum_j f_j\,\dd_X(X_j)
\longmapsto\sum_j f_j \dd_x(X_j)
\end{gather}
is injective and hence it defines an isomorphism
\begin{gather}
\Omega^1(k[X_1,\dots, X_m]) \cong k[X_1,\dots, X_m] \dd_x\big(k[X_1,\dots, X_m]\big) .
\end{gather}
\end{subequations}
As $\Omega^1(A)$ is given by the quotient of $\Omega^1(k[x_1,\dots, x_n])$
by the ideal $(J_A,\dd_x J_A)$ generated by $J_A$ and $\dd_x J_A$,
and likewise $\Omega^1(B) = \Omega^1(k[X_1,\dots, X_m])/(J_B,\dd_XJ_B)$ with
$J_B := J_A\cap k[X_1,\dots,X_m]$, see, e.g., \cite[Section~16.1]{Eis:com},
the required isomorphism is induced by the following diagram with exact rows
\begin{gather*}
\xymatrix@C=1.5em{
0 \ar[r] & (J_B,\dd_XJ_B)\ar[d]_-{\cong} \ar[r] & \Omega^1(k[X_1,\ldots, X_m]) \ar[r]\ar[d]_-{\cong} & \Omega^1(B)\ar[r]\ar@{-->}[d]^-{\cong} & 0 \\
0 \ar[r] & (J_B,\dd_xJ_B) \ar[r] & k[X_1,\ldots, X_m] \dd_x\big(k[X_1,\ldots, X_m]\big) \ar[r] & B\,\dd_A(B) \ar[r] & 0 ,
}
\end{gather*}
where the solid vertical arrows arise from \eqref{iso.Ka.pol}.
\end{proof}

Let us consider now the Atiyah sequence~\eqref{general.Atiyah}
for the {\em classical} Hopf fibration $\bbS^3\to\bbS^2$
and the K\"ahler differential calculi on
$A= \mathcal{O}\big(\bbS^3\big)$ and $H=\mathcal{O}(U(1))$, i.e.,
\begin{gather}\label{eqn:KaehlerAtiyahclassical}
\xymatrix{
0\ar[r]&\Omega^1_\mathrm{hor}(A) \ar[r]& \Omega^1(A) \ar[r]^-{\overline{\mathrm{ver}}} & A\otimes \mathfrak{h}^\vee \ar[r] & 0.
}
\end{gather}
The vector space $\mathfrak{h}^{\vee} = H^+/ (H^+ )^2$ is the algebraic
cotangent space of $U(1)$ at the unit element. Its dual is the vector space underlying
the Lie algebra of $U(1)$, i.e., the vector space $\mathfrak{h} :=\Der_{\epsilon}(H)$ of derivations
relative to $\epsilon\colon H\to\bbC$.
Recall that $X\in \Der_{\epsilon}(H)$ is a linear map $X\colon H\to \bbC$
such that $X(h\,h^\prime) = X(h) \epsilon(h^\prime) + \epsilon(h) X(h^\prime)$,
for all $h,h^\prime\in H$.
It is easily checked that $\mathfrak{h} = \Der_{\epsilon}(H) \cong \bbC$
is one-dimensional. Furthermore, since the generators $z = 2z_1z_2^\ast$,
$z^\ast = 2z_1^\ast z_2$ and $x = z_1^\ast z_1 - z_2^\ast z_2$
of $B = A^{\coH} \cong \OO\big(\bbS^2\big)$ (cf.\ Lemma~\ref{lem:2sphere}) satisfy
the conditions of Lemma~\ref{lem.Ka}, the horizontal forms in~\eqref{eqn:KaehlerAtiyahclassical} are
\begin{gather}\label{hor.Hopf}
\Omega^1_\mathrm{hor}(A) = A\Omega^1(B)A = A\Omega^1(B) = \Omega^1(B)A ,\qquad\text{with}\quad
\Omega^1(B) = \frac{\ker\mu_B}{ (\ker\mu_B )^2} .
\end{gather}
As $(A,\delta,\rho)$ is a $(K,H)$-bicomodule algebra, it follows that~\eqref{eqn:KaehlerAtiyahclassical} is a short exact sequence of $(K,H)$-bicomodule left $A$-modules. Applying the $2$-cocycle
deformation functor from \cite[Proposition~2.4]{Brain2} or
\cite[Proposition~2.25(ii)]{Aschieri}, we obtain
the short exact sequence of $(K,H)$-bicomodule left $A_\theta$-modules
\begin{gather}\label{eqn:KaehlerAtiyahdeformed}
\xymatrix{
0\ar[r]&\Omega^1_\mathrm{hor}(A)_\theta \ar[r]& \Omega^1(A)_\theta \ar[r]^-{\overline{\mathrm{ver}}}
& \big( A\otimes \mathfrak{h}^\vee\big)_\theta \ar[r] & 0.}
\end{gather}

A similar construction applies to the deformed Hopf fibration
described by the $(K,H)$-bicomodule algebra $(A_\theta,\rho,\delta)$ from
Section~\ref{subsec:deformation}. The key point
is that $A_\theta$ is {\em braided commutative} (see, e.g.,~\cite{Barnes}) in the sense of
\begin{subequations}\label{eqn:Rmatrix}
\begin{gather}
a\star_\theta a^\prime = R_\theta(a^\prime \sw{-1}\otimes a \sw{-1})~a^\prime\sw{0} \star_\theta a\sw{0} ,
\end{gather}
for all $a,a^\prime\in A_\theta$, with cotriangular structure
$R_\theta : K\otimes K\to\bbC$ given by
\begin{gather}
R_\theta(t_{\mathbf{m}}\otimes t_{\mathbf{m}^\prime}) = \sigma_\theta(t_{\mathbf{m}}\otimes t_{\mathbf{m}^\prime})^{-2} = \exp\big({-}2\pi {\rm i} \mathbf{m}^\mathrm{T}\Theta \mathbf{m}^\prime\big) .
\end{gather}
\end{subequations}
The product map $\mu_{A_\theta} \colon A_\theta \otimes A_\theta \to A_{\theta}$
is an algebra homomorphism when $A_\theta\otimes A_\theta$ is endowed
with the {\em braided tensor algebra} structure $(a\otimes a^\prime) (\widetilde{a}\otimes \widetilde{a}^\prime) :=
R_\theta(\widetilde{a}\sw{-1}\otimes a^\prime\sw{-1})
(a\star_\theta \widetilde{a}\sw{0})\otimes(a^\prime\sw{0}\star_\theta\widetilde{a}^\prime)$.
Consequently, $\ker\mu_{A_\theta}\subseteq A_\theta\otimes A_\theta$
is an ideal and the {\em deformed K\"ahler forms} may be defined as
\begin{gather*}
\Omega^1(A_\theta) := \frac{\ker \mu_{A_\theta}}{(\ker\mu_{A_\theta})^2} .
\end{gather*}
The vertical lift in \eqref{eqn:universalAtiyah} is an algebra homomorphism with respect
to the braided tensor algebra structures. Hence, it maps
$(\ker\mu_{A_\theta})^2$ to $(A_\theta\otimes H^+)^2 = A_\theta\otimes (H^+)^2$.
In analogy to~\eqref{eqn:KaehlerAtiyahclassical}, we obtain the quotient
short exact sequence
\begin{gather}\label{eqn:KaehlerAtiyah}
\xymatrix{
0\ar[r]&\Omega^1_\mathrm{hor}(A_\theta) \ar[r]& \Omega^1(A_\theta) \ar[r]^-{\overline{\mathrm{ver}}_\theta}
& A_\theta \otimes \mathfrak{h}^\vee \ar[r] & 0.
}
\end{gather}
\begin{propo}\label{propo:KaehlerAtiyahiso}
The isomorphism $\varphi_\theta\colon A_\theta\otimes A_\theta \to (A\otimes A)_\theta$
given in \eqref{eqn:varphimap} descends to the
$(K,H)$-bicomodule left $A_\theta$-module isomorphism
\begin{gather}\label{eqn:barvarphimap}
\overline{\varphi}_\theta\colon \ \Omega^1(A_\theta) \longrightarrow\Omega^1(A)_\theta , \qquad
a \dd_\theta a^\prime \longmapsto a\star_\theta\dd a^\prime .
\end{gather}
This defines an isomorphism of short exact sequences
\begin{gather}\label{eqn:KaehlerAtiyahiso}\begin{split}&
\xymatrix{
0 \ar[r] & \ar[d]_-{\overline{\varphi}_\theta }\Omega^1_{\mathrm{hor}}(A_\theta) \ar[r]& \ar[d]_-{\overline{\varphi}_\theta} \Omega^1(A_\theta) \ar[r]^-{\overline{\mathrm{ver}}_\theta} & A_\theta\otimes \mathfrak{h}^{\vee} \ar[d]^-{\id}\ar[r] & 0\\
0 \ar[r] & \Omega^1_{\mathrm{hor}}(A)_\theta \ar[r]& \Omega^1(A)_\theta \ar[r]^-{\overline{\mathrm{ver}}} & \big(A\otimes \mathfrak{h}^{\vee}\big)_\theta \ar[r] & 0
}\end{split}
\end{gather}
between \eqref{eqn:KaehlerAtiyah} and \eqref{eqn:KaehlerAtiyahdeformed}.
\end{propo}

\begin{proof}
Using the $2$-cocycle property \eqref{eqn:cocyclecondition} of $\sigma_\theta$,
it can be shown that the
map $\varphi_\theta\colon A_\theta\otimes A_\theta \to (A\otimes A)_\theta$
given in~\eqref{eqn:varphimap} is an algebra isomorphism with respect
to the following algebra structures: As above, the domain $A_\theta\otimes A_\theta$
is endowed with the braided tensor algebra structure, i.e.,
\begin{gather*}
(a\otimes a^\prime) (\widetilde{a}\otimes \widetilde{a}^\prime) :=
R_\theta(\widetilde{a}\sw{-1}\otimes a^\prime\sw{-1})
(a\star_\theta \widetilde{a}\sw 0)\otimes(a^\prime \sw 0\star_\theta\widetilde{a}^\prime) .
\end{gather*}
The codomain $(A\otimes A)_\theta$ is endowed with the $2$-cocycle
deformation of the usual tensor algebra structure on $A\otimes A$,
i.e.,
\begin{gather*}
(a\otimes a^\prime)\star_\theta (\widetilde{a}\otimes \widetilde{a}^\prime) :=
\sigma_\theta(a \sw{-1} a^\prime\sw{-1}\otimes \widetilde{a}\sw{-1} \widetilde{a}^\prime\sw{-1})
 (a\sw 0 \widetilde{a} \sw0)\otimes (a^\prime \sw 0 \widetilde{a}^\prime\sw 0) .
\end{gather*}
As a consequence, $\varphi_\theta$ restricts to an
isomorphism $\varphi_\theta \colon (\ker\mu_{A_\theta})^2 \to (\ker \mu_A)^2_\theta$,
which implies that it descends to the claimed isomorphism $\overline{\varphi}_\theta$
between the quotient modules. The explicit expression for $\overline{\varphi}_\theta$
given in \eqref{eqn:barvarphimap} follows from the computation
\begin{align*}
\nn \overline{\varphi}_\theta(a \dd_\theta a^\prime) &=\varphi_\theta(a\otimes a^\prime - a\star_\theta a^\prime\otimes \1\big) = \sigma_\theta(a\sw{-1}\otimes a^\prime\sw{-1}) \big(a\sw{0}\otimes a^\prime\sw{0} - a\sw{0} a^\prime\sw{0}\otimes\1\big)\\
&= \sigma_\theta(a\sw{-1}\otimes a^\prime\sw{-1}) a\sw{0} \dd(a^\prime\sw{0}) = a\star_\theta \dd a^\prime ,
\end{align*}
where in the second step we used~\eqref{eqn:varphimap} and~\eqref{eqn:starproduct},
and in the last step we used that $\dd$ is a~$K$-co\-module map.
The statement about short exact sequences follows by using also
Proposi\-tion~\ref{propo:universalAtiyahiso}.
\end{proof}
\begin{cor}\label{cor:deformedhorizonalB}
$\Omega^1_\mathrm{hor}(A_\theta) = A_\theta \Omega^1(B) A_\theta = A_\theta \Omega^1(B) = \Omega^1(B) A_\theta$ with $\Omega^1(B)$ the K\"ahler $1$-forms on $B= A^{\coH}_\theta \subseteq A_\theta$.
\end{cor}
\begin{proof}
By \eqref{hor.Hopf} and invertibility of the $2$-cocycle $\sigma_\theta$,
it follows that
\begin{gather*}
\Omega^1_\mathrm{hor}(A)_\theta = A_\theta \star_\theta \Omega^1(B)
\star_\theta A_\theta = A_\theta \star_\theta\Omega^1(B) = \Omega^1(B) \star_\theta A_\theta ,
\end{gather*}
where $\star_\theta$ denotes the deformed left and right
$A_\theta$-module structures on $\Omega^1(A)_\theta$.
Our claim then follows from the isomorphism $\overline{\varphi}_\theta$
given in~\eqref{eqn:barvarphimap}, because it maps bijectively between
$a (\dd_\theta b) a^\prime \in A_\theta \Omega^1(B) A_\theta\subseteq
\Omega^1(A_\theta)$ and $a\star_\theta (\dd b)\star_\theta a^\prime\in A_\theta
\star_\theta \Omega^1(B) \star_\theta A_\theta\subseteq \Omega^1(A)_\theta$.
\end{proof}

\subsection{Connections}
The aim of this section is to characterize connections with respect to the K\"ahler differential calculi
for both the classical Hopf fibration $\bbS^3\to \bbS^2$ and the
deformed Hopf fibration. It will be shown that they are equivalent in a suitable sense.

We first consider the classical Hopf fibration $(A,\delta)$ from Section~\ref{subsec:classical}.
By Definition \ref{def:connectiongeneralcalculus}, the set of connections is the set of
splittings of \eqref{eqn:KaehlerAtiyahclassical} or equivalently the set of
connection forms
\begin{gather*}
\Con(A,\delta) := \big\{ \overline{\omega} \in \Hom^H\big(\mathfrak{h}^\vee, \Omega^1(A)\big) \colon
\overline{\ver}\circ \overline{\omega} =\1\otimes\id \big\} ,
\end{gather*}
where $\Hom^H$ denotes the set of right $H$-comodule morphisms.
Let us recall that $\mathfrak{h}= \Der_{\epsilon}(H)\cong\bbC$
is $1$-dimensional and choose a basis $X\in\mathfrak{h}$, e.g., the linear map
\begin{gather}\label{eqn:X}
X \colon \ H\longrightarrow\bbC ,\qquad t^n \longmapsto n.
\end{gather}
Let $\chi \in\mathfrak{h}^\vee$ be the dual basis defined by $\langle \chi,X\rangle =1$.
Composing $\overline{\mathrm{ver}}$ with the evaluation map $\langle -,X\rangle \colon \mathfrak{h}^\vee\to\bbC$
defines the morphism
\begin{gather*}
\overline{\mathrm{ver}}^X \colon \ \Omega^1(A) \longrightarrow A~,
\omega \longmapsto \overline{\mathrm{ver}}^X(\omega):= \langle \overline{\mathrm{ver}}(\omega),X\rangle .
\end{gather*}
As the right adjoint $H$-coaction on $\mathfrak{h}^\vee$ is trivial,
it follows that
\begin{gather*}
\Con(A,\delta) \cong \big\{\omega \in \Omega^1(A)^{\coH}\colon
\overline{\mathrm{ver}}^X(\omega) = \1 \big\} .
\end{gather*}
The bijection is given explicitly by $\overline{\omega}(\chi) =\omega$,
for the dual basis vector $\chi \in\mathfrak{h}^\vee$.
Analogously, the set of connections for the deformed Hopf fibration $(A_\theta,\delta)$
from Section~\ref{subsec:deformation} is the set of splittings of~\eqref{eqn:KaehlerAtiyah}
or equivalently the set of connection forms
\begin{align}
\nn \Con(A_\theta,\delta) :=& \big\{ \overline{\omega}_\theta \in \Hom^H\big(\mathfrak{h}^\vee, \Omega^1(A_\theta)\big) \colon \overline{\mathrm{ver}}_\theta\circ \overline{\omega}_\theta =\1 \otimes \id \big\}\\
 \cong& \big\{\omega_\theta \in \Omega^1(A_\theta)^{\coH}
\colon \overline{\mathrm{ver}}_\theta^X (\omega_\theta) = \1 \big\} .\label{eqn:deformedCon}
\end{align}
Notice that every connection on $(A,\delta)$ and also every connection on $(A_\theta,\delta)$
is strong because of~\eqref{hor.Hopf} and Corollary~\ref{cor:deformedhorizonalB}.
\begin{propo}\label{propo:Coniso}
The isomorphism $ \overline{\varphi}_\theta\colon \Omega^1(A_\theta)\to \Omega^1(A)_\theta$
from Proposition~{\rm \ref{propo:KaehlerAtiyahiso}} defines a bijection
\begin{gather*}
\overline{\varphi}_\theta \colon \ \Con(A_\theta,\delta) \longrightarrow \Con(A,\delta), \qquad \omega_\theta\longmapsto \overline{\varphi}_\theta (\omega_\theta) .
\end{gather*}
\end{propo}
\begin{proof}
Commutativity of the diagram \eqref{eqn:KaehlerAtiyahiso} implies
$\overline{\mathrm{ver}}^X\big(\overline{\varphi}_\theta(\omega_\theta)\big)
=\overline{\mathrm{ver}}^X_\theta( \omega_\theta) = \1$, i.e.,
$\overline{\varphi}_\theta( \omega_\theta) \in\Con(A,\delta)$.
The inverse map $ \Con(A, \delta)\to \Con(A_\theta,\delta)$
is given by $\omega \mapsto \overline{\varphi}_\theta^{-1}(\omega)$.
\end{proof}

\begin{remm}Similar results for connections on modules were proven in \cite{AschieriCon,BarnesNAG}.
\end{remm}

Our next aim is to refine the result of Proposition~\ref{propo:Coniso}
by using more explicit features of the example under investigation.
Using \cite[Section~16.1]{Eis:com},
the module of K\"ahler $1$-forms on $A$ can be computed as
\begin{gather}\label{eqn:classicalKaehler}
\Omega^1(A) = \frac{A \dd z_1 \oplus A \dd z_2 \oplus A \dd z_1^\ast \oplus A \dd z_2^\ast}{\langle
\dd z_1^\ast z_1 + z_1^\ast \dd z_1 + \dd z_2^\ast z_2 + z_2^\ast \dd z_2 \rangle} ,
\end{gather}
where the right $A$-action is defined by $s a:= a s$, for all $a\in A$
and $s\in A \dd z_1 \oplus A \dd z_2 \oplus A \dd z_1^\ast \oplus A \dd z_2^\ast$.
The differential $\dd \colon A \to \Omega^1(A)$ is specified by mapping
each generator $z_1$, $z_2$, $z_1^\ast$, $z_2^\ast$ of $A$ to the corresponding generator
$\dd z_1$, $\dd z_2$, $\dd z_1^\ast$, $\dd z_2^\ast$ of~\eqref{eqn:classicalKaehler}
and the Leibniz rule. Moreover, the vertical lift
$\overline{\mathrm{ver}}^X \colon \Omega^1(A) \to A$ is given by
\begin{gather}\label{eqn:verXexplicit}
\overline{\mathrm{ver}}^X (a \dd a^\prime) =
a a^\prime\sw 0 X(a^\prime \sw 1) ,
\end{gather}
for all $a,a^\prime\in A$.

Using braided commutativity of the
deformed Hopf fibration $(A_\theta,\delta,\rho)$, a similar computation shows that the
module of K\"ahler $1$-forms on $A_\theta$ reads as
\begin{gather}\label{eqn:deformedKaehler}
\Omega^1(A_\theta) = \frac{A_\theta \dd_\theta z_1 \oplus A_\theta \dd_\theta z_2 \oplus
A_\theta \dd_\theta z_1^\ast \oplus A_\theta \dd_\theta z_2^\ast}{\langle
\dd_\theta z_1^\ast z_1 + z_1^\ast \dd_\theta z_1 + \dd_\theta z_2^\ast z_2 + z_2^\ast \dd_\theta z_2 \rangle} ,
\end{gather}
where the right $A_\theta$-action is defined using the cotriangular structure
\eqref{eqn:Rmatrix} by $s a := R_\theta(a\sw{-1}\otimes s\sw{-1}) a\sw 0 s\sw 0$, for
all $a \in A_\theta$ and $s\in A_\theta \dd_\theta z_1 \oplus A_\theta \dd_\theta z_2 \oplus
A_\theta \dd_\theta z_1^\ast \oplus A_\theta\,\dd_\theta z_2^\ast $.
The differential $\dd_\theta \colon A _\theta\to \Omega^1(A_\theta)$ is specified by mapping
each generator $z_1$, $z_2$, $z_1^\ast$, $z_2^\ast$ of $A_\theta$ to the corresponding generator
$\dd_\theta z_1$, $\dd_\theta z_2$, $\dd_\theta z_1^\ast$, $\dd_\theta z_2^\ast$ of~\eqref{eqn:deformedKaehler}
and the Leibniz rule. The vertical lift
$\overline{\mathrm{ver}}_\theta^X \colon \Omega^1(A_\theta) \to A_\theta$ is given by
\begin{gather}\label{eqn:verthetaXexplicit}
\overline{\mathrm{ver}}_\theta^X (a\,\dd_\theta a^\prime) = a\star_\theta a^\prime \sw 0 X(a^\prime \sw 1) ,
\end{gather}
for all $a,a^\prime\in A_\theta$.
\begin{lem}\label{lem:Conexistence}
The $1$-form $\omega^0 := z_1^\ast \dd z_1 + z_2^\ast \dd z_2 \in\Omega^1(A)$
defines a $K$-coinvariant connection on the classical Hopf fibration $(A,\delta)$.
Hence, there exists a bijection
\begin{gather*}
\Con(A,\delta) \cong \big\{\omega^0 + \alpha \colon \alpha\in \Omega^1(B)\big\} .
\end{gather*}
Analogously, the $1$-form $\omega_\theta^0 :=
z_1^\ast \dd_\theta z_1 + z_2^\ast \dd_\theta z_2 \in\Omega^1(A_\theta)$
defines a $K$-coinvariant connection on the deformed Hopf fibration $(A_\theta,\delta)$.
Hence, there exists a bijection
\begin{gather*}
\Con(A_\theta,\delta) \cong \big\{\omega_\theta^0 + \alpha \colon \alpha\in \Omega^1(B)\big\} .
\end{gather*}
\end{lem}
\begin{proof}The one-form $\omega^0$ is clearly $H$-coinvariant and $K$-coinvariant.
Using~\eqref{eqn:verXexplicit} and~\eqref{eqn:X}, it is found that
\begin{gather*}
\overline{\mathrm{ver}}^X\big(\omega^0\big) = z_1^\ast z_1 X(t) +
z_2^\ast z_2 X(t) = z_1^\ast z_1 + z_2^\ast z_2 = \1 .
\end{gather*}
The bijection is a consequence
of $\Omega^1_\mathrm{hor}(A)^\coH= \Omega^1(B)$ being the K\"ahler forms on $B$, cf.~\eqref{hor.Hopf}.

Analogously, $\omega_\theta^0$ is $H$-coinvariant and $K$-coinvariant.
Using \eqref{eqn:verthetaXexplicit} and~\eqref{eqn:X}, it is found that
\begin{gather*}
\overline{\mathrm{ver}}_\theta^X\big(\omega_\theta^0\big) = z_1^\ast\star_\theta z_1 + z_2^\ast \star_\theta z_2 = \1 ,
\end{gather*}
where the last step uses~\eqref{eqn:starsphererelation}. The bijection is a consequence
of $\Omega^1_{\mathrm{hor}}(A_\theta)^\coH = \Omega^1(B)$ being the K\"ahler forms on~$B$, cf.\ Corollary~\ref{cor:deformedhorizonalB}.
\end{proof}

\begin{propo}\label{prop:conpreservation}
Expressed in the form of Lemma~{\rm \ref{lem:Conexistence}},
the bijection from Proposition~{\rm \ref{propo:Coniso}} reads explicitly as
\begin{gather*}
\overline{\varphi}_\theta \colon \ \Con(A_\theta,\delta) \longrightarrow \Con(A,\delta) ,\qquad
\omega^0_\theta + \alpha \longmapsto \omega^0 + \alpha .
\end{gather*}
In other words, the affine map $\overline{\varphi}_\theta$ preserves
the points $\omega^0_\theta$ and $\omega^0$, and its
linear part given by the identity map $\id\colon \Omega^1(B)\to \Omega^1(B)$.
\end{propo}
\begin{proof}From the explicit expression for $\overline{\varphi}_\theta$ given in~\eqref{eqn:barvarphimap},
it is easy to confirm that $\overline{\varphi}_\theta (\omega_\theta^0) = \omega^0$ and that
$\overline{\varphi}_\theta (b \dd_\theta b^\prime) = b \dd b^\prime$,
for all $b,b^\prime\in B$, which proves our claim.
\end{proof}
\begin{remm}\label{rem:connectionbijection}
Similarly to Remark~\ref{rem:nococycle}, the result in Proposition~\ref{prop:conpreservation}
is stronger than the isomorphisms that can be obtained by using the
general theory of $2$-cocycle deformations from \cite{Brain,Brain2} and \cite{Aschieri}.
Concretely, the property that $\alpha$ is unchanged by the isomorphism
$\omega^0_\theta + \alpha \mapsto \omega^0 + \alpha$ is a
particular feature of the example under investigation.
\end{remm}

\subsection{Gauge transformations}
We describe a notion of (infinitesimal) gauge transformations
for both the deformed and the classical Hopf fibration, together
with their actions on connections. We shall show that
the identification of connections from Propositions
\ref{propo:Coniso} and \ref{prop:conpreservation}
intertwines between the deformed and classical gauge transformations.
In other words, the theory of connections and their infinitesimal
gauge transformations on the deformed Hopf fibration
$(A_\theta,\delta)$ is equivalent to that on the classical
Hopf fibration $\bbS^3\to \bbS^2$.

Let us note that our notion of gauge transformations will be formalized
by {\em braided derivations} and hence it makes
explicitly use of the braided commutativity of $A_\theta$.
There also exists a more flexible
concept of gauge transformations
given by right $H$-comodule left $B$-module
automorphisms $f\colon A\to A$ satisfying $f(\1)=\1$, see, e.g.,~\cite{Brz:translation}.
Note that such $f$ are {\em not} required to be algebra homomorphisms.
However, for commutative principal comodule algebras, this definition does
not recover the usual concept of gauge transformations in classical geometry,
in contrast to our more special approach by braided derivations.
As a last remark, let us note that it would also be possible to describe
finite gauge transformations by using the noncommutative mapping
spaces from \cite{Barnes}. This is technically more involved
and will not be discussed here.

We describe presently the case of infinitesimal gauge transformations
of the deformed Hopf fibration $(A_\theta,\delta)$, which includes
the classical case $(A,\delta)$ by setting $\theta=0$.
Consider the associated left $B$-module $E_{A_\theta}(\mathfrak{h})
=A_\theta\Box_H \mathfrak{h}$, where the Lie algebra $\mathfrak{h} = \Der_\epsilon(H)$
is endowed with the adjoint left $H$-coaction, which is trivial as $U(1)$
is Abelian. Hence, $E_{A_\theta}(\mathfrak{h}) \cong B\otimes \mathfrak{h}$
as left $B$-modules and, using the basis element $X\in\mathfrak{h} $
from \eqref{eqn:X}, any element $\zeta \in E_{A_\theta}(\mathfrak{h}) $
can be written as $\zeta = b\otimes X $, for a unique $b\in B$.
\begin{defi}\label{def:infgaugetrafos}
An {\em infinitesimal gauge transformation} of the deformed Hopf fibration
$A_\theta$ is an element of the left $B$-module $E_{A_\theta}(\mathfrak{h})$
associated to $\mathfrak{h} = \Der_\epsilon(H)$.
The action of an infinitesimal gauge transformation $\zeta = b\otimes X \in E_{A_\theta}(\mathfrak{h})$
on $A_\theta$ is defined as
\begin{gather}\label{eqn:gaugeaction}
(-)\,{{}_{\theta}\triangleleft} \, \zeta \colon \ A_\theta \longrightarrow A_\theta ,\qquad
a \longmapsto a\,{{}_{\theta}\triangleleft}\, \zeta := a\sw 0\star_\theta b X(a\sw 1) .
\end{gather}
\end{defi}
\begin{remm}
It is easily checked that the action \eqref{eqn:gaugeaction} of
infinitesimal gauge transformations satisfies
\begin{gather*}
(a\star_\theta a^\prime)\,{{}_{\theta}\triangleleft}\, \zeta =
a\star_\theta \big(a^\prime \,{{}_{\theta}\triangleleft}\, \zeta \big) + R_{\theta}(\zeta\sw{-1} \otimes a^\prime\sw{-1})
\big(a\,{{}_{\theta}\triangleleft}\, \zeta\sw 0\big) a^\prime\sw{0} ,
\end{gather*}
for all $a,a^\prime\in A_\theta$ and $\zeta \in E_{A_\theta}(\mathfrak{h}) $, i.e.,
$E_{A_\theta}(\mathfrak{h})$ acts on~$A_\theta$ from the right by braided derivations. In particular,
this action preserves the left $B$-module structure on~$A_\theta$, i.e.,
$(b^\prime \star_\theta a)\,{{}_{\theta}\triangleleft} \, \zeta = b^\prime \star_\theta (a \,{{}_{\theta}\triangleleft}\, \zeta)$,
for all $b^\prime \in B$, $a\in A_\theta$ and $\zeta \in E_{A_\theta}(\mathfrak{h})$.
\end{remm}

\begin{remm}The infinitesimal gauge transformations from Definition~\ref{def:infgaugetrafos}
are a generalization of the analogous concept in ordinary differential geometry. Given a principal $G$-bundle
$P\to M$ over a manifold $M$, infinitesimal gauge transformations are given by the space of sections of
the vector bundle $VP/G\to P/G\cong M$ of vertical tangent vectors modulo~$G$.
Using fundamental vector fields,
the latter is isomorphic to the space of sections of the adjoint bundle
$P\times_{\mathrm{ad}}\mathfrak{g}\to M$,
which in our noncommutative example is given by $E_{A_\theta}(\mathfrak{h}) $. The action in
\eqref{eqn:gaugeaction} is the evident generalization of the usual action of infinitesimal gauge
transformations from ordinary differential geometry.
\end{remm}
The action \eqref{eqn:gaugeaction} can be extended to the differential graded algebra
$(\Omega^\bullet(A_\theta),\wedge_\theta,\dd_\theta)$ of deformed K\"ahler forms.
The latter is obtained by the {\em braided exterior algebra} of the symmetric
$A_\theta$-bimodule $\Omega^1(A_\theta)$, i.e., the wedge-product $\wedge_\theta$
satisfies the braided graded commutativity property
$\lambda \wedge_\theta \lambda^\prime = (-1)^{n\,m}~R_\theta(\lambda^\prime\sw{-1}\otimes \lambda\sw{-1})~
\lambda^\prime\sw 0\wedge_\theta \lambda\sw 0$, for all homogeneous
$\lambda \in\Omega^n(A_\theta)$ and $\lambda^\prime \in \Omega^m(A_\theta)$,
while the differential satisfies the graded Leibniz rule\footnote{Our variant of the
graded Leibniz rule is obtained by thinking of $\dd_\theta$ as acting from the right.
This is a convenient choice of convention for the present work,
because all other operations (e.g., covariant derivatives and gauge transformations) also act from the right.}
\begin{gather}\label{eqn:gradedLeibniz}
\dd_\theta(\lambda \wedge_\theta \lambda^\prime) = \lambda\wedge_\theta (\dd_\theta\lambda^\prime) +
(-1)^{m} (\dd_\theta\lambda)\wedge_\theta \lambda^\prime ,
\end{gather}
for all $\lambda\in\Omega^\bullet(A_\theta)$ and all homogeneous $\lambda^\prime \in \Omega^m(A_\theta)$.
\begin{remm}\label{abusenotation}
Using that the isomorphism $\Omega^1(A_\theta)\cong \Omega^1(A)_\theta$ from
Proposition~\ref{propo:KaehlerAtiyahiso} is an isomorphism
of $(K,H)$-bicomodule $A_\theta$-bimodules, we obtain a canonical isomorphism
of differential graded algebras
\begin{gather*}
(\Omega^\bullet(A_\theta),\wedge_\theta,\dd_\theta) \cong (\Omega^\bullet(A),\wedge,\dd)_\theta ,
\end{gather*}
where the right hand side is the $2$-cocycle deformation of the
differential graded algebra of undeformed K\"ahler forms, see, e.g., \cite[Proposition~3.17]{Brain2}.
By a convenient abuse of notation, we shall often suppress this isomorphism in what follows, i.e.,
we simply identify $\lambda \wedge_\theta \lambda^\prime\in \Omega^\bullet(A_\theta)$
with the deformed wedge-product $\lambda\wedge_\theta \lambda^\prime =
\sigma_{\theta}(\lambda\sw{-1}\otimes\lambda^\prime\sw{-1})
\lambda\sw 0\wedge \lambda^\prime \sw 0\in \Omega^\bullet(A)_\theta$ and $\dd_\theta \lambda \in
\Omega^\bullet(A_\theta)$ with $\dd \lambda\in \Omega^\bullet(A)_\theta$.
\end{remm}

For every $\zeta=b\otimes X \in E_{A_\theta}(\mathfrak{h})$, we
define the contraction map $\iota_\zeta^\theta\colon \Omega^1(A_\theta)\to A_\theta $ by setting
\begin{gather*}
\iota^\theta_\zeta(a\, \dd_\theta a^\prime) := a\star_\theta (a^\prime \,{{}_{\theta}\triangleleft}\,\zeta )
= \overline{\ver}_\theta^X (a \dd_\theta a^\prime)\star_\theta b ,
\end{gather*}
for all $a,a^\prime \in A_\theta$. Because $\overline{\ver}_\theta^X$ is a left $B$-module
morphism, $\iota^\theta_\zeta$ preserves the left
$B$-module structures too, i.e., $\iota^\theta_\zeta( b^\prime \star_\theta \lambda) = b^\prime \star_\theta
\iota_{\zeta}^\theta(\lambda)$, for all $ b^\prime \in B$ and $\lambda\in\Omega^1(A_\theta)$.
This map can be extended to the whole of $\Omega^\bullet(A_\theta)$ as a braided anti-derivation, i.e.,
\begin{gather*}
\iota^\theta_\zeta(\lambda\wedge_\theta\lambda^\prime) =
\lambda\wedge_\theta \iota^\theta_{\zeta}(\lambda^\prime) +
(-1)^m R_\theta(\zeta\sw{-1}\otimes \lambda^\prime\sw{-1})
\iota^\theta_{\zeta\sw 0} (\lambda)\wedge_\theta \lambda^\prime \sw 0 ,
\end{gather*}
for all $\lambda\in \Omega^\bullet(A_\theta)$ and all homogeneous
$\lambda^\prime\in\Omega^m(A_\theta)$.
\begin{defi}
The action of an infinitesimal gauge transformation
$\zeta\in E_{A_\theta}(\mathfrak{h})$ on a general form
$\lambda\in \Omega^\bullet(A_\theta)$ is defined by Cartan's magic formula
\begin{gather*}
\lambda \,{{}_{\theta}\triangleleft}\,\zeta :=
\dd_\theta \iota^\theta_\zeta(\lambda) + \iota^\theta_\zeta(\dd_\theta\lambda) .
\end{gather*}
\end{defi}

\begin{lem}\label{eqn:undeformedgaugetrafos}
Let $\omega_\theta \in \Con(A_\theta,\delta) \subseteq \Omega^1(A_\theta)$
be a deformed connection, cf.~\eqref{eqn:deformedCon}.
Then $ \omega_\theta \,{{}_{\theta}\triangleleft}\,\zeta = \dd_\theta b$,
for all $\zeta = b\otimes X\in E_{A_\theta}(\mathfrak{h})$.
\end{lem}
\begin{proof}
Notice that $\iota^\theta_\zeta(\omega_\theta) = \overline{\mathrm{ver}}_\theta^X(\omega_\theta)\star_\theta b
= \1\star_\theta b=b$ by~\eqref{eqn:deformedCon}.
Recalling~\eqref{eqn:deformedKaehler}, it is found that
$\omega_\theta = \sum\limits_{i=1}^{4} a^i \dd_\theta z_i$, where $z_3=z_1^\ast$ and $z_4 = z_2^\ast$,
for some $a^i\in A_\theta$.
As $\omega_\theta$ is $H$-coinvariant by hypothesis and $\delta(\dd_\theta z_i) = \dd_\theta z_i\otimes t_i$
with $t_1=t_2=t$ and $t_3=t_4=t^\ast$,
it follows that $\delta(a^i) = a^i \otimes t_i^\ast$. Hence,
\begin{align*}
\nn \omega_\theta \,{{}_{\theta}\triangleleft}\,\zeta &= \dd_\theta \iota^\theta_\zeta(\omega_\theta)
+ \iota^\theta_\zeta(\dd_\theta \omega_\theta) = \dd_\theta b
-\sum_{i=1}^4 \iota^\theta_\zeta\big(\dd_\theta a^i \wedge_\theta \dd_\theta z_i\big)\\
\nn &= \dd_\theta b - \sum_{i=1}^4 \big( \dd_\theta a^i z_i X(t_i) -a^i X(t_i^\ast) \dd_\theta z_i \big) b\\
\nn &= \dd_\theta b - \sum_{i=1}^4 \big( \dd_\theta a^i~z_i + a^i \dd_\theta z_i \big) b ~X(t_i)\\
&= \dd_\theta b - \dd_\theta\left(\sum_{i=1}^4 a^i \star_\theta z_i X(t_i)\right) b=
\dd_\theta b - \big(\dd_\theta \overline{\mathrm{ver}}_\theta^X(\omega_\theta )\big) b =\dd_\theta b ,
\end{align*}
where the third line uses \eqref{eqn:X} and the last step follows by \eqref{eqn:deformedCon}.
\end{proof}

Setting $\theta=0$, we obtain the usual infinitesimal gauge transformations
$E_A(\mathfrak{h}) = A\Box_H \mathfrak{h}$ for the classical
Hopf fibration $(A,\delta)$, together with their action $\triangleleft$
on the undeformed K\"ahler forms $(\Omega^\bullet(A),\wedge,\dd)$.
\begin{remm}\label{rem:defundefgauge}
As the left $H$-coaction on $\mathfrak{h}$ is trivial,
it follows that the left $B$-module isomorphism $L_{\mathfrak{h}}\colon E_A(\mathfrak{h})\to E_{A_\theta}(\mathfrak{h})$ from Proposition~\ref{propo:moduleequivalence} is the identity map, i.e.,
undeformed and deformed gauge transformations are identified `on the nose'.
\end{remm}
\begin{propo}
The bijection $\overline{\varphi}_\theta\colon \Con(A_\theta,\delta)\to \Con(A,\delta)$
from Proposition \ref{propo:Coniso} preserves infinitesimal gauge transformations, i.e.,
\begin{gather*}
\overline{\varphi}_\theta\big( \omega_\theta \,{{}_{\theta}\triangleleft}\,\zeta \big) =
\overline{\varphi}_\theta(\omega_\theta) \,{\triangleleft}\,\zeta ,
\end{gather*}
for all $\zeta \in E_A(\mathfrak{h}) = E_{A_\theta}(\mathfrak{h})$
and $\omega_\theta \in \Con(A_\theta,\delta)$.
\end{propo}
\begin{proof}
Recalling the explicit expression for $\overline{\varphi}_\theta$ from \eqref{eqn:barvarphimap},
it is found that $\overline{\varphi}_\theta\big( \omega_\theta \,{{}_{\theta}\triangleleft}\,\zeta \big)
= \overline{\varphi}_\theta(\dd_\theta b) = \dd b = \overline{\varphi}_\theta(\omega_\theta)
\,{\triangleleft}\,\zeta$.
\end{proof}

\section{Associated gauge transformations and connections}\label{sec:associatedcon}
Let $(V,\rho)\in\lcom H$ be a left $H$-comodule and consider
the associated left $B$-modules $E_A(V)= A\Box_H V$
and $E_{A_\theta}(V) = A_\theta\Box_H V$ from Section \ref{sec:associatedmodules}.
By Proposition~\ref{propo:moduleequivalence}, there exists a natural
left $B$-module isomorphism $L_V\colon E_{A}(V) \to E_{A_\theta}(V)$.
We shall show that this isomorphism {\em does not} intertwine between the
actions of gauge transformations and connections on associated modules.
Physically speaking, this means that the coupling of gauge and matter fields
on the deformed Hopf fibration $(A_\theta,\delta)$ is different to that on
the classical Hopf fibration $\bbS^3\to\bbS^2$.

Let us start with the associated gauge transformations.
Using \eqref{eqn:gaugeaction}, there exists an induced action
of $E_{A_\theta}(\mathfrak{h})$ on $E_{A_\theta}(V)$ which is given by
\begin{gather}\label{eqn:associatedgaugeactions}
(a\otimes v)\,{{}_{\theta}\triangleleft}\,\zeta :=
(a\,{{}_{\theta}\triangleleft}\,\zeta)\otimes v = a\sw 0 \star_\theta b \otimes v X(a\sw 1)
= a\star_\theta b \otimes X(v\sw{-1}) v\sw 0 ,
\end{gather}
for all $a\otimes v\in E_{A_\theta}(V)$ and $\zeta = b\otimes X \in E_{A_{\theta}}(\mathfrak{h})$,
with the last equality following from the definition of the cotensor product~\eqref{eqn:cotensorproduct}.
Note that each $(-)\,{{}_{\theta}\triangleleft}\, \zeta\colon E_{A_\theta}(V)\to E_{A_\theta}(V)$
is a left $B$-module endomorphism. Hence, it defines a notion of infinitesimal gauge transformations
that coincides with the analogous concepts from the more standard projective module approach to
noncommutative gauge theory, see, e.g.,~\cite{Landi}.
Setting $\theta=0$, we obtain a similar expression without the $\star_\theta$-product, i.e.,
\begin{gather}\label{eqn:associatedundeformedgaugeactions}
(a\otimes v)\,{\triangleleft} \zeta = a b \otimes X(v\sw{-1}) v\sw 0 ,
\end{gather}
for all $a\otimes v\in E_{A}(V)$ and $\zeta = b\otimes X \in E_{A}(\mathfrak{h})$.
Let us also recall from Remark~\ref{rem:defundefgauge} that
$E_{A_{\theta}}(\mathfrak{h}) = E_{A}(\mathfrak{h})$ are identified with
the identity map $L_{\mathfrak{h}} = \id$.
\begin{propo}\label{propo:gaugeintertwine}
Suppose that the left $H$-coaction of $(V,\rho)\in \lcom H$ is non-trivial, i.e.,
there exists $v\in V$ such that $\rho(v) \neq \1_H \otimes v$. Then the
left $B$-module isomorphism $L_V\colon E_{A}(V) \to E_{A_\theta}(V)$
from Proposition~{\rm \ref{propo:moduleequivalence}} does not
intertwine between the actions of classical~\eqref{eqn:associatedundeformedgaugeactions}
and deformed~\eqref{eqn:associatedgaugeactions} infinitesimal gauge transformations.
\end{propo}
\begin{proof}Because the $H$-coaction on $V$ is by hypothesis non-trivial,
there exists a non-zero homogeneous element
$a\otimes v\in A^{((m+n,-m),n)} \otimes V^n$ with $n\neq 0$.
Consider any infinitesimal gauge transformation $\zeta =
b\otimes X\in E_{A_{\theta}}(\mathfrak{h}) = E_{A}(\mathfrak{h})$ with
$b\in A^{((m^\prime,-m^\prime),0)} $ a non-zero homogeneous element with $m^\prime\neq 0$.
Using~\eqref{eqn:Lcomponents} and~\eqref{eqn:associatedundeformedgaugeactions},
we compute
\begin{subequations}\label{eqn:gaugetrafocorrections}
\begin{gather}
L_{V}\big((a\otimes v)\,{\triangleleft}\,\zeta\big)
= {\rm e}^{\pi {\rm i} \theta (m+m^\prime)n} a b\otimes X(v\sw{-1}) v\sw 0 .
\end{gather}
Using also~\eqref{eqn:associatedgaugeactions} and~\eqref{eqn:starproductdecomposition}, we compute
\begin{gather}
L_V(a\otimes v) \,{{}_{\theta}\triangleleft}\,\zeta
 = {\rm e}^{\pi {\rm i}\theta (m-m^\prime)n} a b\otimes X(v\sw{-1}) v\sw 0 .
\end{gather}
\end{subequations}
The phase factors differ because $n\neq 0$ and $m^\prime\neq 0$, which proves our claim.
\end{proof}
\begin{remm}\label{rem:gaugeintertwine}
The result of Proposition~\ref{propo:gaugeintertwine}
that $L_V$ does not intertwine between classical and
deformed gauge transformations can also be understood
from the following more heuristic argument: Gauge transformations
of {\em left} $B$-modules act from the {\em right}
(cf.~\eqref{eqn:associatedgaugeactions} where~$b$ is to the right of $a$) and hence they are in general not intertwined by the {\em left}
$B$-module isomorphism~$L_V$.
\end{remm}

Let us now discuss the covariant derivatives \eqref{cov.derivgeneral}
and their corresponding associated connections \eqref{eqn:Nablageneral}.
We fix any connection form $\omega = \omega^0+\alpha\in\Con(A,\delta)$
on the classical Hopf fibration (cf.\ Lemma \ref{lem:Conexistence})
and determine the corresponding connection form on the deformed Hopf fibration
$(A_\theta,\delta)$ via the bijection given in Proposition \ref{prop:conpreservation},
i.e., $\omega_\theta = \omega^0_\theta + \alpha\in \Con(A_\theta,\delta)$ with the same $\alpha$.
For later convenience, we consider the canonical lift of the
covariant derivative $\overline{D}\colon A\to \Omega^1_{\mathrm{hor}}(A)$
to the graded subalgebra $(\Omega^\bullet_{\mathrm{hor}}(A),\wedge) \subseteq
(\Omega^\bullet(A),\wedge)$ of horizontal K\"ahler forms.
Note that $\Omega^\bullet_{\mathrm{hor}}(A)=\Omega^\bullet(B)A$ because
of~\eqref{hor.Hopf} and graded commutativity of $\wedge$. This yields a~linear map $\overline{D} \colon \Omega^\bullet_{\mathrm{hor}}(A)\to \Omega^{\bullet+1}_{\mathrm{hor}}(A)$,
which explicitly reads as
\begin{gather}\label{eqn:undeformedcovder}
\overline{D}(\lambda) = \dd \lambda - \lambda\sw 0 \wedge \omega X(\lambda\sw 1)= \dd \lambda - \lambda\sw 0 \wedge \big(\omega^0+\alpha\big) X(\lambda\sw 1) ,
\end{gather}
for all $\lambda\in \Omega^\bullet_{\mathrm{hor}}(A) $. By construction,
$\overline{D}$ satisfies the graded Leibniz rule
\begin{gather*}
\overline{D}(\beta\wedge \lambda) = (-1)^m \dd \beta \wedge \lambda + \beta \wedge \overline{D}(\lambda) ,
\end{gather*}
for all $\beta \in \Omega^\bullet(B)$ and all homogeneous $\lambda\in \Omega^m_\mathrm{hor}(A)$.
(Compare this with the similar graded Leibniz rule for $\dd$ in~\eqref{eqn:gradedLeibniz}.)
The canonical lift $\overline{D}_\theta \colon \Omega^\bullet_{\mathrm{hor}}(A_\theta)
\to \Omega^{\bullet+1}_{\mathrm{hor}}(A_\theta)$ of the corresponding deformed covariant derivative
to the graded subalgebra $(\Omega^\bullet_{\mathrm{hor}}(A_\theta),\wedge_\theta)
\subseteq (\Omega^\bullet(A_\theta),\wedge_\theta)$ of horizontal deformed K\"ahler forms explicitly reads as
\begin{gather}\label{eqn:deformedcovder}
\overline{D}_\theta(\lambda ) = \dd_\theta \lambda - \lambda\sw 0 \wedge_\theta \omega_\theta~
X(\lambda\sw 1) =\dd_\theta \lambda - \lambda\sw 0 \wedge_\theta \big(\omega_\theta^0+\alpha\big)
X(\lambda\sw 1) ,
\end{gather}
for all $\lambda\in \Omega^\bullet_{\mathrm{hor}}(A_\theta)$. By construction,
$\overline{D}_\theta$ satisfies the graded Leibniz rule
\begin{gather*}
\overline{D}_\theta(\beta\wedge_\theta \lambda) = (-1)^m \dd_\theta \beta \wedge_\theta \lambda +
\beta \wedge_\theta \overline{D}_\theta(\lambda) ,
\end{gather*}
for all $\beta \in \Omega^\bullet(B)$
and all homogeneous $\lambda\in \Omega^m_\mathrm{hor}(A_\theta)$.
Notice that $\Omega^\bullet_{\mathrm{hor}}(A_\theta)=\Omega^\bullet(B)A_\theta$
because of Corollary~\ref{cor:deformedhorizonalB} and braided graded commutativity of $\wedge_\theta$.

Regarding both $\Omega^\bullet_{\mathrm{hor}}(A) = \Omega^\bullet(B)A$
and $\Omega^\bullet_{\mathrm{hor}}(A_\theta) = \Omega^\bullet(B) A_\theta$
as graded left $\Omega^\bullet(B)$-modules,
we obtain in analogy to Proposition \ref{propo:moduleequivalence} a graded left
$\Omega^\bullet(B)$-module isomorphism
$L^\bullet\colon \Omega^\bullet_{\mathrm{hor}}(A)\to \Omega^\bullet_{\mathrm{hor}}(A_\theta)$
by setting
\begin{gather}\label{eqn:Lbullet}
L^\bullet(\lambda) :={\rm e}^{\pi {\rm i} \theta mn} \lambda ,
\end{gather}
for all homogeneous elements $\lambda\in
\Omega^{\bullet}_{\mathrm{hor}}(A)^{((m+n,-m),n)} \subseteq \Omega^\bullet_{\mathrm{hor}}(A)$.
\begin{propo}\label{propo:conintertwine}
Suppose that the chosen connection form $\omega = \omega^0 + \alpha \in \Con(A,\delta)$
is not $K$-coinvariant. Then the left $\Omega^\bullet(B)$-module isomorphism
$L^\bullet\colon \Omega^\bullet_{\mathrm{hor}}(A)\to \Omega^\bullet_{\mathrm{hor}}(A_\theta)$
given in \eqref{eqn:Lbullet} {\em does not} intertwine between
the classical~\eqref{eqn:undeformedcovder} and the deformed~\eqref{eqn:deformedcovder} covariant derivatives.
\end{propo}
\begin{proof}
Let us decompose $\alpha$ as $\alpha = \sum_{m^\prime\in \bbZ} \alpha_{m^\prime}$, where
$\alpha_{m^\prime} \in \Omega^{1}(B)^{((m^\prime,-m^\prime),0)}$,
and note that by hypothesis there exists $m^\prime\neq 0$ such that $\alpha_{m^\prime}\neq 0$.
For any homogeneous $\lambda\in \Omega^{\bullet}_{\mathrm{hor}}(A)^{((m+n,-m),n)} $ with $n\neq 0$,
we find after a short calculation using the identification explained in
Remark~\ref{abusenotation} and the analogue of~\eqref{eqn:starproductdecomposition} for
deformed wedge-products
\begin{gather}\label{eqn:Lbulletnonpreservation}
(L^\bullet)^{-1}\big(\overline{D}_\theta L^\bullet(\lambda)\big)
= \dd \lambda - \lambda\sw 0 \wedge \bigg(\omega^0 + \sum_{m^\prime\in\bbZ} {\rm e}^{-2\pi {\rm i} \theta m^\prime n} \alpha_{m^\prime} \bigg) X(\lambda\sw 1) ,
\end{gather}
which is different to $\overline{D}(\lambda)$ given in~\eqref{eqn:undeformedcovder}
because $\alpha_{m^\prime}\neq 0$ for some $m^\prime \neq 0$.
\end{proof}
\begin{remm}Similarly to Remark~\ref{rem:gaugeintertwine}, this feature can also be understood from the
fact that covariant derivatives act from the {\em right} on the {\em left} $\Omega^{\bullet}(B)$-modules
$\Omega^\bullet_{\mathrm{hor}}(A)$ and $\Omega^\bullet_{\mathrm{hor}}(A_\theta) $
and hence they are in general not intertwined by the {\em left} $\Omega^{\bullet}(B)$-module isomorphism~$L^\bullet$. An exception is the case of a $K$-coinvariant
connection form $\omega = \omega^0 + \alpha \in \Con(A,\delta)$, i.e., $\alpha \in \Omega^1(B)^{((0,0),0)}$.
\end{remm}
\begin{remm}
As a side-remark, we note the following expressions for the classical and deformed
{\em curvatures}, which are defined as the squares of the corresponding covariant derivatives.
In the classical case, we obtain by a direct calculation using \eqref{eqn:undeformedcovder}
\begin{align*}
 \overline{D}^2(\lambda) &= \dd\big( \dd \lambda - \lambda\sw 0 \wedge \omega X(\lambda\sw 1)\big) -\big( \dd \lambda\sw 0- \lambda\sw 0_{\sw 0} \wedge \omega X(\lambda\sw 0_{\sw 1})\big)\wedge \omega X(\lambda\sw 1)\\
& =-\lambda\sw 0 \wedge \dd\omega X(\lambda\sw 1) + \lambda\sw 0 \wedge\omega\wedge\omega
X(\lambda \sw{1}_{\sw 1}) X(\lambda\sw{1}_{\sw 2})\\
& =-\lambda\sw 0 \wedge \dd\omega X(\lambda\sw 1) ,
\end{align*}
where in the last equality we used that by graded commutativity
$\omega\wedge\omega = -\omega\wedge\omega =0$.
In the deformed case, we obtain by a similar direct calculation using~\eqref{eqn:deformedcovder}
\begin{gather*}
\overline{D}_\theta^2(\lambda) = -\lambda\sw 0 \wedge_\theta \dd_\theta\omega_\theta ~X(\lambda\sw 1) .
\end{gather*}
Here we also use that $\omega_\theta = \omega^0_\theta +\alpha$, where
$\omega^0_\theta$ is $K$-coinvariant and $\alpha\in \Omega^1(B)$, hence
by braided graded commutativity
$\omega_\theta \wedge_\theta \omega_\theta =
- R_{\theta}((\omega_{\theta})\sw{-1}\otimes (\omega_{\theta})\sw{-1})~(\omega_{\theta})\sw0\wedge_\theta (\omega_{\theta})\sw 0 = - \omega_\theta \wedge_\theta \omega_\theta =0 $.
Under the identification explained in Remark~\ref{abusenotation},
the underlying deformed curvature form~$\dd_\theta\omega_\theta$ coincides with the
classical curvature form~$\dd\omega$. However, by a similar argument
as above, the actions of these curvature forms on, respectively,
$\Omega^\bullet_{\mathrm{hor}}(A_\theta)$ and $\Omega^\bullet_{\mathrm{hor}}(A)$
are in general not intertwined by the isomorphism $L^\bullet$.
\end{remm}

Let now $(V,\rho)\in \lcom H$ be any left $H$-comodule and consider the corresponding associated connections $\overline{\nabla}\colon E_A(V)\to \Omega^1(B)\otimes_B E_{A}(V)$
and $\overline{\nabla}_\theta\colon E_{A_\theta}(V)\to \Omega^1(B)\otimes_B E_{A_\theta}(V)$.
The latter are obtained concretely by composing, respectively,
$\overline{D}\otimes\id\colon A\Box_H V\to \Omega^1_\mathrm{hor}(A)\Box_H V$
and $\overline{D}_\theta \otimes\id \colon A_\theta \Box_H V\to \Omega^1_\mathrm{hor}(A_\theta)\Box_H V$
with the inverses of the corresponding left $B$-module isomorphisms
\begin{gather*}
\nn \Omega^1(B)\otimes_B E_{A}(V) \longrightarrow \Omega^1_\mathrm{hor}(A)\Box_H V ,
 \beta\otimes_B (a\otimes v) \longmapsto \beta a\otimes v ,\\
 \Omega^1(B)\otimes_B E_{A_\theta}(V) \longrightarrow \Omega^1_\mathrm{hor}(A_\theta)\Box_H V ,
 \beta\otimes_B (a\otimes v)\longmapsto \beta\star_\theta a\otimes v .
\end{gather*}
These isomorphisms are compatible with the isomorphisms $L_V$ in Proposition~\ref{propo:moduleequivalence} and $L^\bullet$ in~\eqref{eqn:Lbullet} in the sense that the diagram{\samepage
\begin{gather*}
\xymatrix@C=4em{
\ar[d]_-{\id\otimes_B L_V} \Omega^1(B)\otimes_B E_{A}(V) \ar[r] & \Omega^1_\mathrm{hor}(A)\Box_H V \ar[d]^-{L^\bullet \otimes \id}\\
\Omega^1(B)\otimes_B E_{A_\theta}(V) \ar[r] &\Omega^1_\mathrm{hor}(A_\theta)\Box_H V
}
\end{gather*}
commutes. The following result is then an immediate consequence of Proposition~{\rm \ref{propo:conintertwine}}.}
\begin{cor}

Suppose that the left $H$-coaction of $(V,\rho)\in \lcom H$ is non-trivial
and that the connection form $\omega \in \Con(A,\delta)$
is not $K$-coinvariant. Then the left $B$-module isomorphism
$L_V\colon E_{A}(V)\to E_{A_\theta}(V)$ from Proposition~{\rm \ref{propo:moduleequivalence}}
{\em does not} intertwine between the classical and the deformed associated connections,
i.e., $\overline{\nabla}_\theta\circ L_V \neq (\id\otimes_B L_V) \circ\overline{\nabla}$.
\end{cor}

\begin{ex}
Let us consider the explicit example where
$(V,\rho) = \bbC^{(n)} := (\bbC, \rho^{(n)})$ is the $1$-dimensional
irreducible left $H$-comodule given by $\rho^{(n)}\colon c\mapsto t^n \otimes c$, for some $n\in\bbZ$.
The associated left $B$-modules $E_{A}(\bbC^{(n)}) = \bigoplus_{m\in\bbZ} A^{((m+n,-m),n)} $
and $E_{A_\theta} (\bbC^{(n)}) = \bigoplus_{m\in\bbZ} A^{((m+n,-m),n)}$ describe
a matter field of charge $n\in\bbZ$ on the classical and deformed Hopf fibrations respectively.
Every deformed gauge transformation by $\zeta \in E_{A_\theta}(\mathfrak{h})$
can be interpreted in terms of an undeformed gauge transformation by defining an {\em effective
gauge transformation parameter} $\zeta_\theta^{(n)} \in E_A(\mathfrak{h})$ according to
\begin{gather*}
\lambda \,{\triangleleft}\,\zeta_\theta^{(n)} := L_V^{-1}\big(L_V(\lambda)\,{{}_{\theta}\triangleleft}\,\zeta\big) ,
\end{gather*}
for all $\lambda\in E_{A} (\bbC^{(n)})$. To compute $\zeta_\theta^{(n)}$ explicitly,
we decompose $\zeta$ as $\zeta = \sum_{m^\prime\in\bbZ} \zeta_{m^\prime}$, where
$\zeta_{m^\prime}\in A^{((m^\prime,-m^\prime),0)}$, and obtain from~\eqref{eqn:gaugetrafocorrections} the expression
\begin{gather*}
\zeta_\theta^{(n)} = \sum_{m^\prime\in\bbZ} {\rm e}^{-2\pi {\rm i} \theta m^\prime n} \zeta_{m^\prime} ,
\end{gather*}
where we also used Remark \ref{rem:defundefgauge} to identify
$E_{A_{\theta}}(\mathfrak{h}) = E_{A}(\mathfrak{h})$ with
the identity map $L_{\mathfrak{h}} = \id$.

Analogously, the deformed covariant derivative can be interpreted in terms of an undeformed covariant derivative by defining
\begin{gather*}
\overline{D}_\theta^{(n)}(\lambda) := (L^\bullet)^{-1}\big(\overline{D}_\theta L^\bullet(\lambda)\big) ,
\end{gather*}
for all $\lambda\in E_{A} (\bbC^{(n)})$. Decomposing
the $1$-form $\alpha\in\Omega^1(B)$ in \eqref{eqn:deformedcovder} as
$\alpha = \sum_{m^\prime\in\bbZ} \alpha_{m^\prime} $, where
$\alpha_{m^\prime}\in \Omega^1(B)^{((m^\prime,-m^\prime),0)}$,
and using \eqref{eqn:Lbulletnonpreservation}, we obtain
\begin{gather*}
\overline{D}_\theta^{(n)}(\lambda) = \dd \lambda - \lambda\sw 0 \wedge \big(\omega^0+\alpha_\theta^{(n)}\big)
X(\lambda\sw 1) ,
\end{gather*}
where the {\em effective gauge potential} is given by
\begin{gather*}
\alpha_\theta^{(n)} = \sum_{m^\prime\in\bbZ} {\rm e}^{-2\pi {\rm i} \theta m^\prime n} \alpha_{m^\prime} .
\end{gather*}
Summing up, the physical effect of the studied deformation on charged matter fields
is that they experience effective gauge transformations and effective covariant derivatives
that depend on both the charge $n\in\bbZ$ and on the deformation parameter $\theta\in\bbR$.
\end{ex}

\section{\label{sec:homotopy}Outlook: Towards an explanation via homotopy equivalence}
In this section, we construct a suitable type of homotopy equivalence between the classical and deformed
Hopf fibrations. The concept of homotopy equivalence presented here is slightly different
from the one proposed by Kassel and Schneider \cite{KS} in the sense that
for our purposes it is more convenient to work with a larger algebra
of complex-valued functions on the real line~$\bbR$ than the one
given by the polynomial algebra $\bbC[y]$ in one generator~$y$. In particular, we
would like the exponential functions $y\mapsto {\rm e}^{{\rm i} c y}$, for $c\in\bbR$,
to be contained in this algebra in order to obtain a homotopy equivalence that relates
the exponentiated deformation parameter $q = {\rm e}^{2\pi {\rm i} \theta}$ to the classical value~$1$.

Let us fix any $\ast$-subalgebra $\overline{\OO}(\bbR) \subseteq \mathrm{Fun}(\bbR,\bbC)$
of the $\ast$-algebra of complex-valued functions on the real line that contains the
monomials $y^n\in \overline{\OO}(\bbR)$, for $n\geq 0$,
and also the exponential functions ${\rm e}^{{\rm i}c y}\in \overline{\OO}(\bbR)$, for $c\in\bbR$.
For example, we could choose the algebra of analytic func\-tions~$C^\omega(\bbR,\bbC)$,
the algebra of smooth func\-tions~$C^\infty(\bbR,\bbC)$ or the algebra of continuous func\-tions~$C^0(\bbR,\bbC)$. Note that the polynomial algebra $\bbC[y]$ used in \cite{KS}
does not satisfy our criteria because it does not contain the exponential functions.
Consider the two evaluation maps $\mathrm{ev}_{p}\colon \overline{\OO}(\bbR)\to \bbC$,
$f\mapsto f(p)$, for $p=0,1$, and note that they are algebra homomorphisms.
The following definition generalizes \cite[Definition~2.3]{KS}
to the new concept of $\overline{\OO}(\bbR)$-homotopy equivalence.
\begin{defi}\label{def:homotopy}
Let $H$ be a Hopf algebra and $B$ an algebra.
Let $A_0$ and $A_1$ be two principal $H$-comodule
algebras over $B$. We write $A_0 \sim A_1$
if there exists a principal $H\otimes \overline{\OO}(\bbR)$-comodule algebra
$A$ over $B\otimes \overline{\OO}(\bbR)$ with ground ring $\overline{\OO}(\bbR)$
such that ${\mathrm{ev}_p}_\ast (A) := A\otimes_{\overline{\OO}(\bbR)}^{}\bbC\, \cong \, A_p$
are isomorphic as $H$-comodule algebras, for $p=0,1$.\footnote{Here ${\mathrm{ev}_p}_\ast$ denotes
the usual change of base ring construction along the algebra maps
$\mathrm{ev}_p\colon \overline{\OO}(\bbR) \to \bbC$. It is given concretely by the relative tensor product
${\mathrm{ev}_p}_\ast (A) = A\otimes_{\overline{\OO}(\bbR)}^{}\bbC$ where the
left $\overline{\OO}(\bbR)$-module structure on $\bbC$ is defined by $\overline{\OO}(\bbR) \otimes \bbC
\stackrel{\mathrm{ev_p}\otimes \id }{\longrightarrow} \bbC\otimes\bbC
\stackrel{\mu_\bbC}{\longrightarrow}\bbC$.}

{\em $\overline{\OO}(\bbR)$-homotopy equivalence of principal $H$-comodule
algebras} over $B$ is the equivalence relation $\approx$ generated by~$\sim$.
\end{defi}

In the following, we let $H = \OO(U(1))$ and $B = \OO\big(\bbS^2\big)$ as in Section~\ref{subsec:classical}.
We denote the classical Hopf fibration from Section~\ref{subsec:classical} by $A_0$
and the deformed Hopf fibration, with a fixed deformation parameter $\theta\in\bbR$,
from Section \ref{subsec:deformation} by~$A_1$. Our aim is to construct
for this scenario a principal $H\otimes \overline{\OO}(\bbR)$-comodule algebra
$A$ over $B\otimes \overline{\OO}(\bbR)$ with ground ring $\overline{\OO}(\bbR)$, as
in Definition~\ref{def:homotopy}, that implements an $\overline{\OO}(\bbR)$-homotopy equivalence
between the classical Hopf fibration $A_0$ and the deformed Hopf fibration~$A_1$.
This requires a few preparations. Let us first note that
the change of base ring functor $(-)\otimes \overline{\OO}(\bbR)\colon
\mathscr{M}_{\bbC} \to \mathscr{M}_{\overline{\OO}(\bbR)}$
maps free algebras over $\bbC$ to the corresponding
free algebras over $\overline{\OO}(\bbR)$ because it is monoidal, i.e.,
$(V\otimes \overline{\OO}(\bbR)) \otimes_{\overline{\OO}(\bbR)} (W\otimes \overline{\OO}(\bbR))
\cong (V\otimes W)\otimes \overline{\OO}(\bbR)$ for all $V,W\in\mathscr{M}_{\bbC}$,
and it preserves all small colimits. Because of the latter property, it also follows that
$(-)\otimes \overline{\OO}(\bbR)$ maps finitely presented algebras over $\bbC$
to the corresponding finitely presented algebras over~$\overline{\OO}(\bbR)$.
This implies that $H\otimes \overline{\OO}(\bbR)$ is the commutative
$\ast$-algebra over~$\overline{\OO}(\bbR)$ generated by one generator
$t$ that satisfies the circle relation $t^\ast t=\1$, and that
$B\otimes \overline{\OO}(\bbR)$ is the commutative $\ast$-algebra
over $\overline{\OO}(\bbR)$ generated by two generators
$z$ and $x$ that satisfy $x^\ast=x$ and the $2$-sphere relation $z^\ast z +x^2=\1$. In other words,
$H\otimes \overline{\OO}(\bbR) $ and $B\otimes \overline{\OO}(\bbR)$ are described
by the same formulae as the ones for $H$ and $B$ in
Section~\ref{subsec:classical}, however the base ring is now $ \overline{\OO}(\bbR)$ instead of $\bbC$.

Let us denote by $Q\in \overline{\OO}(\bbR)$ the function given by
\begin{gather}\label{eqn:Qmap}
Q \colon \ \bbR \longrightarrow \bbC ,\qquad y\longmapsto Q(y) = {\rm e}^{2\pi {\rm i} \theta y} .
\end{gather}
Note that $Q(0) =1$ and $Q(1) = q = {\rm e}^{2\pi {\rm i}\theta}$. We define~$A$ to be the noncommutative $\ast$-algebra over $ \overline{\OO}(\bbR)$ generated by two generators $z_{1}$ and $z_2$ satisfying the
commutation relations
\begin{gather}\label{eqn:Qcomrel}
z_1  z_1^\ast = z_1^\ast   z_1 ,\qquad  z_2  z_2^\ast = z_2^\ast  z_2 ,\qquad
z_1  z_2 = Q z_2  z_1 ,\qquad z_1  z_2^\ast = Q^{-1} z_2^\ast   z_1 ,
\end{gather}
and the $3$-sphere relation
\begin{gather}\label{eqn:sphereQrel}
z_1^\ast z_1 + z_2^\ast z_2 = \1 .
\end{gather}
Compare this to~\eqref{eqn:ConnesLandirelations} and~\eqref{eqn:starsphererelation}.
We further endow~$A$ with a right $H\otimes \overline{\OO}(\bbR) $-comodule structure
$\delta\colon A\to A\otimes_{\overline{\OO}(\bbR) } (H\otimes \overline{\OO}(\bbR)) $
by using the same formulae~\eqref{eqn:Hcoaction} for the generators as for the case where the base ring is~$\bbC$.
\begin{lem}The $H\otimes \overline{\OO}(\bbR)$-comodule algebra $(A,\delta)$ described above is a principal
comodule algebra with subalgebra of coinvariants isomorphic to $B\otimes \overline{\OO}(\bbR)$.
\end{lem}
\begin{proof}
The proof is analogous to the one for algebras over $\bbC$, cf.\ Sections~\ref{subsec:classical} and~\ref{subsec:deformation} and also~\cite{BrzezinskiSitarz}.
\end{proof}

We can now show that this principal comodule algebra
implements an $\overline{\OO}(\bbR)$-homotopy equivalence
between the classical Hopf fibration $A_0$ and the
deformed Hopf fibration $A_1$.
\begin{propo}\label{propo:homotopy}
There exist isomorphisms of $H$-comodule algebras
${\mathrm{ev}_p}_\ast(A) \cong A_p$, for $p=0,1$.
Hence, the classical Hopf fibration $A_0$ is $\overline{\OO}(\bbR)$-homotopy equivalent
to the deformed Hopf fibration $A_1$ for any value of the deformation parameter $\theta\in\bbR$.
\end{propo}
\begin{proof}
Using the map $\bbC\to \overline{\OO}(\bbR)$, $c\mapsto c$ that assigns
to a complex number the corresponding constant function on $\bbR$, we
can consider $A$ as an algebra over $\bbC$.
The homomorphism of $\bbC$-algebras
$A \to {\mathrm{ev}_p}_\ast(A) = A\otimes_{\overline{\OO}(\bbR)}\bbC$,
$a\mapsto a\otimes_{\overline{\OO}(\bbR)} 1$ is surjective as
$a\otimes_{\overline{\OO}(\bbR)} c = ac \otimes_{\overline{\OO}(\bbR)} 1 $.
The kernel of this map is the two-sided ideal of the $\bbC$-algebra $A$ generated
by $(f -f(p)) \1\in A$, for all $f\in \overline{\OO}(\bbR)$. Hence, ${\mathrm{ev}_p}_\ast(A)$
is isomorphic to the quotient $\bbC$-algebra $A\big/\big((f-f(p)) \1 \colon f\in \overline{\OO}(\bbR)\big)$,
i.e., the evaluation of all coefficient functions in $\overline{\OO}(\bbR)$
at $p\in\bbR$. Recalling~\eqref{eqn:Qmap},
\eqref{eqn:Qcomrel} and~\eqref{eqn:sphereQrel}, we obtain that
${\mathrm{ev}_0}_\ast(A)$ is isomorphic to the classical $3$-sphere
due to $Q(0) =1$ and that ${\mathrm{ev}_1}_\ast(A)$ is
isomorphic to the deformed $3$-sphere with deformation parameter
$\theta$ as $Q(1) = q={\rm e}^{2\pi {\rm i}\theta}$. These isomorphisms are clearly
compatible with the $H$-coactions.
\end{proof}

We finish this section by explaining why we believe that
Proposition \ref{propo:homotopy} is the conceptual reason
for the results in Section \ref{subsec:associated} stating that the associated module functors
for the classical and deformed Hopf fibration are naturally isomorphic.
The following argument is inspired by \cite[Remark 2.4(4)]{KS}.
Let us first note that the underlying vector space of the Hopf algebra $H = \OO(U(1))$ admits a decomposition
$H = \bigoplus_{n\in\bbZ} C_n$, where $C_n =\bbC \,t^n$ is the
($1$-dimensional) vector space spanned by the $n$-th power of the generator $t$.
(Recall that $t^{-1} = t^\ast$.)
As $\Delta(t^n)= t^n\otimes t^n$,
each $C_n$ is a left $H$-comodule via the coaction $\Delta\colon C_n\to H\otimes C_n$.
It follows that $A \cong \bigoplus_{n\in\bbZ} A\square_H C_n$
is a direct sum of associated modules $E_A(C_n) = A\square_H C_n$
with $C_n\in \lcomf H $ finite-dimensional. By \cite{BrzHaj:Che},
each $E_A(C_n) $ is a finitely generated projective left $B\otimes \overline{\OO}(\bbR)$-module
and hence it defines an element $[E_A(C_n)]$ of the
zeroth $K$-theory group $K_0\big(B\otimes \overline{\OO}(\bbR)\big)$. Analogously, there exists a decomposition
$A_p \cong \bigoplus_{n\in\bbZ} A_p \square_H C_n$ into a direct sum
of associated modules $E_{A_p}(C_n) = A_p\square_H C_n$
and we obtain elements $[E_{A_p}(C_n)]\in K_0(B)$, for $p=0,1$.
Because of Proposition \ref{propo:homotopy},
we know that $[E_{A_p}(C_n)]\in K_0(B)$ is the image under
\begin{gather*}
K_0({\mathrm{ev}_p}_\ast) \colon \ K_0\big(B \otimes \overline{\OO}(\bbR)\big) \longrightarrow K_0(B)
\end{gather*}
of the element $[E_A(C_n)]\in K_0\big(B\otimes \overline{\OO}(\bbR)\big)$, for $p=0,1$.
The same holds true for the modu\-les~$E_A(V)$, $E_{A_0}(V)$ and
$E_{A_1}(V)$ associated to any finite-dimensional left $H$-comodule
as any such~$V$ decomposes as a finite direct sum of~$C_n$'s.

If we could prove that the $K_0$-groups are invariant under
$\overline{\OO}(\bbR)$-homotopy equivalences, the concrete
result of Section~\ref{subsec:associated}, namely that the associated modules
$E_{A_0}(V)$ and $E_{A_1}(V)$ of the classical and deformed Hopf fibrations
are isomorphic, would follow from the more conceptual argument in this section.
While results in this direction are available for homotopy equivalences described by the polynomial
algebra $\bbC[y]$, see, e.g.,~\cite{Bass}, we are not aware of generalizations
to the setup of $ \overline{\OO}(\bbR) $-homotopy equivalences. We expect that addressing
this question might provide some insights on appropriate choices of
the algebra $\overline{\OO}(\bbR) $, which we have left unspecified above,
besides assuming that it contains exponential functions and polynomials.
Developing a theory of $\overline{\OO}(\bbR) $-homotopy equivalences would be useful
and interesting also for other typical examples in noncommutative geometry
where the deformation parameter appears in an exponentiated form $q= {\rm e}^{2\pi {\rm i}\theta}$.
We hope to come back to this issue in a future work.

\subsection*{Acknowledgments}
We would like to thank Christian Lomp and Alexander Vishik for useful comments
related to this work. We also would like to thank the anonymous referees
for their valuable comments and suggestions that helped us to improve this manuscript.
The research of T.B.~was partially supported by the Polish National Science
Centre grant 2016/21/B/ST1/02438.
A.S.~gratefully acknowledges the financial support of
the Royal Society (UK) through a Royal Society University
Research Fellowship, a Research Grant and an Enhancement Award.

\pdfbookmark[1]{References}{ref}
\LastPageEnding

\end{document}